\documentclass[11pt]{article}
\usepackage[margin=1in]{geometry}
\usepackage{graphicx} 
\usepackage{subcaption}
\usepackage{tikz}
\usepackage{float}
\usetikzlibrary{trees}

\usepackage{amsfonts}
\usepackage{mathtools}

\usepackage{fullpage}
\usepackage{stmaryrd}
\usepackage{soul}
\usepackage{color}
\usepackage{times}
\usepackage{fancyhdr,graphicx,amsmath,amssymb}
\usepackage{amsthm}
\usepackage{bbm}
\usepackage{dsfont}
\usepackage[noend]{algorithmic}
\usepackage[ruled,noend]{algorithm2e}

\DeclareMathOperator{\Pow}{Pow}

\DeclareMathOperator{\argmax}{argmax}

\DeclareMathOperator{\epi}{epi}
\DeclareMathOperator{\supp}{Supp}
\usepackage{amsthm}
\usepackage{xcolor}

\newcommand{\N}{\mathbb N}
\newcommand{\R}{\mathbb R}

\newcommand{\HH}{\mathcal H}

\newtheorem{lem}{Lemma}[section]

\newtheorem{defn}[lem]{Definition}
\newtheorem{thm}[lem]{Theorem}
\newtheorem{rem}[lem]{Remark}
\newtheorem{exm}[lem]{Example}
\newtheorem{prop}[lem]{Proposition}

\newtheorem{assumption}{Assumption}

\numberwithin{equation}{section}

\title{On the sequential convergence of Lloyd's algorithms}

\author{
  Léo Portales\\
  IRIT, TSE, INP Toulouse\\
  \url{leo.portales@irit.fr}
  \and
  Elsa Cazelles\\
  CNRS, IRIT, Université de Toulouse\\
  \url{elsa.cazelles@irit.fr}
  \and
  Edouard Pauwels\\
  TSE, Université Toulouse Capitole\\
  \url{edouard2.pauwels@ut-capitole.fr}
}

\usepackage{capt-of}
\usepackage{bbm}
\usepackage{hyperref}
\hypersetup{
    colorlinks=true,
    linkcolor=blue,
    filecolor=magenta,      
    urlcolor=cyan
}
\begin{document}
\maketitle

\begin{abstract}

	Lloyd's algorithm is an iterative method that solves the quantization problem, i.e. the approximation of a target probability measure by a discrete one, and is particularly used in digital applications.
    This algorithm can be interpreted as a gradient method on a certain quantization functional which is given by optimal transport.
	We study the sequential convergence (to a single accumulation point) for two variants of Lloyd's method: (i) \emph{optimal quantization} with an arbitrary discrete measure and (ii) \emph{uniform quantization} with a uniform discrete measure. For both cases, we prove sequential convergence of the iterates under an analiticity assumption on the density of the target measure. This includes for example analytic densities truncated to a compact semi-algebraic set. The argument leverages the log analytic nature of globally subanalytic integrals, the interpretation of Lloyd's method as a gradient method and the convergence analysis of gradient algorithms under Kurdyka-\L{ojasiewicz} assumptions. 
	As a by-product, we also obtain definability results for more general semi-discrete optimal transport losses such as transport distances with general costs, the max-sliced Wasserstein distance and the entropy regularized optimal transport loss.
\end{abstract}

\section{Introduction}
Quantization aims to approximate a target probability measure $\mu$ on $\R^d$ by a discrete, finitely supported measure. We consider two quantization settings induced by an optimization problem in Wasserstein distance. 
First, optimal quantization consists in approximating the target measure $\mu$ with a discrete probability measure supported on $N$ points $y_{1},\ldots,y_{N}\in \mathbb{R}^{d}$ with weights $\pi\in\Delta_N$, the $N-1$ dimensional unit simplex, by solving the following optimization problem:
\begin{equation}\label{opt_quant_0}
   \underset{y_1,\ldots,y_N\in \mathbb{R}^{d}}{\min} \ \underset{\pi \in \Delta_{N}}{\min} \qquad W_{2}^{2}\left(\mu,\sum_{i=1}^{N} \pi_{i} \delta_{y_{i}}\right).
\end{equation}
Optimal quantization in signal processing was introduced to address the central problem of converting a continuous-time signal into a digital one \cite{ref_signal_opt_quant_2}. It is in particularly used in compressed sensing \cite{ref_signal_opt_quant,ref_signal_opt_quant_3}, but also applies to a wide range of fields including dynamic systems \cite{larson1967optimum} or deep neural networks \cite{zhou2018adaptive}, to reduce computational costs. We refer the reader to \cite{Bookquantization} for a comprehensive introduction to quantization while the review \cite{pages2015introduction} by Pagès provides an overview of its applications in clustering, numerical integration, and finance. Second, uniform quantization consists in approximating $\mu$ with a discrete measure, corresponding to constraints on the optimal quantization weights: 
\begin{equation}
	\label{eq:unif_quantization_0}
	\underset{y_1,\ldots,y_N\in \mathbb{R}^{d}}{\min} \qquad W_{2}^{2}\left(\mu,\frac{1}{N}\sum_{i=1}^{N} \delta_{y_{i}}\right).
\end{equation}
This optimization problem, and more generally constrained quantization, is a very natural extension to \eqref{opt_quant_0} as one might aim to approximate a probability measure $\mu$ with a point cloud or a discrete measure with weights determined by a physical problem, as described in \cite{Bournehexagonalpatterns}. Uniform quantization also commonly sees applications in imaging for image stippling \cite{LebratGournay2D,balzer2009capacity}. In both problem, $W_2^2$ denotes the squared $2$-Wasserstein distance, the optimal transport distance for the squared Euclidean cost between probability measures (see e.g. \cite{refvillanitopics, villanioptoldnew, booksantambrogio}). The Wasserstein distance is well suited in this context since it allows to handle the semi-discrete nature of the quantization task. The main question addressed in this work is that of the asymptotic convergence of the well known Lloyd iterative algorithmic method for problems \eqref{opt_quant_0} and \eqref{eq:unif_quantization_0}. This algorithm may be interpreted as a fixed stepsize gradient descent algorithm. Actually, most\footnote{The unsupervised clustering algorithm based on Self Organizing Map (SOM \cite{cottrell2018self,asan2012introduction}) introduced by Kohonen \cite{kohonen1991self} is a quantization method that does not strictly solves \eqref{opt_quant_0}, although it coincides with Lloyd's in particular cases. SOM algorithm only has convergence guarantees in dimension $1$,  and weaker results are available in higher dimension \cite{cottrell2018self}.
} quantization methods are based on minimizing \eqref{opt_quant_0} or \eqref{eq:unif_quantization_0}, and the vast majority are iterative methods, including many variants of Lloyd's method, or gradient methods. Note that both problems are non-convex and iterative solvers typically find stationary points rather than global minima.In this work we focus on the simplest version of Lloyd's algorithm with fixed stepsize, which is the basis for a large family of quantization methods for which the sequential convergence was not established.

\begin{figure}
    \centering
    \subfloat{{\includegraphics[width=3.7cm]{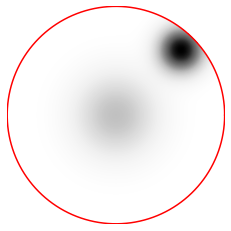}}}
    \subfloat{{\includegraphics[width=4cm]{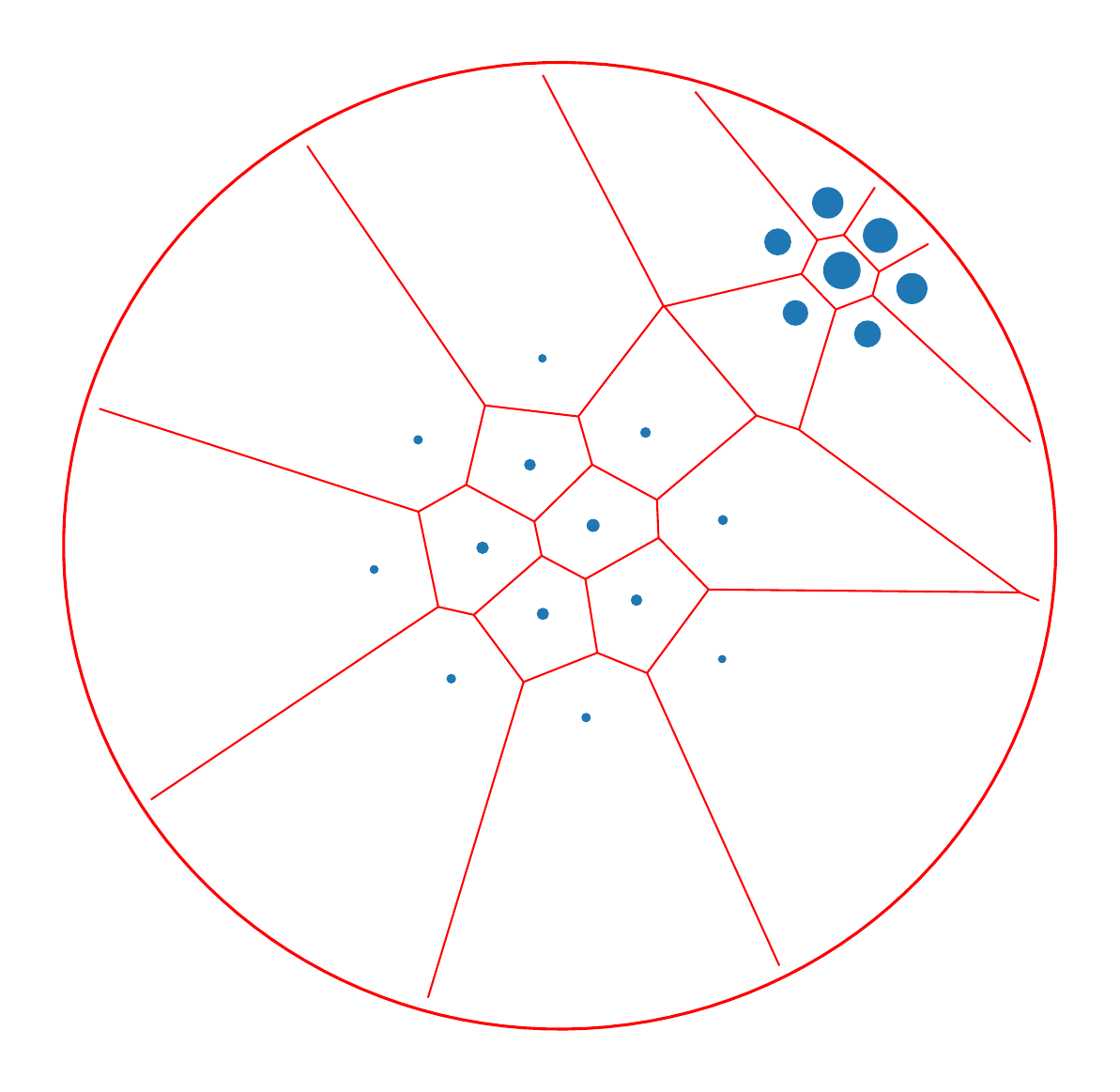} }}
    \subfloat{{\includegraphics[width=4cm]{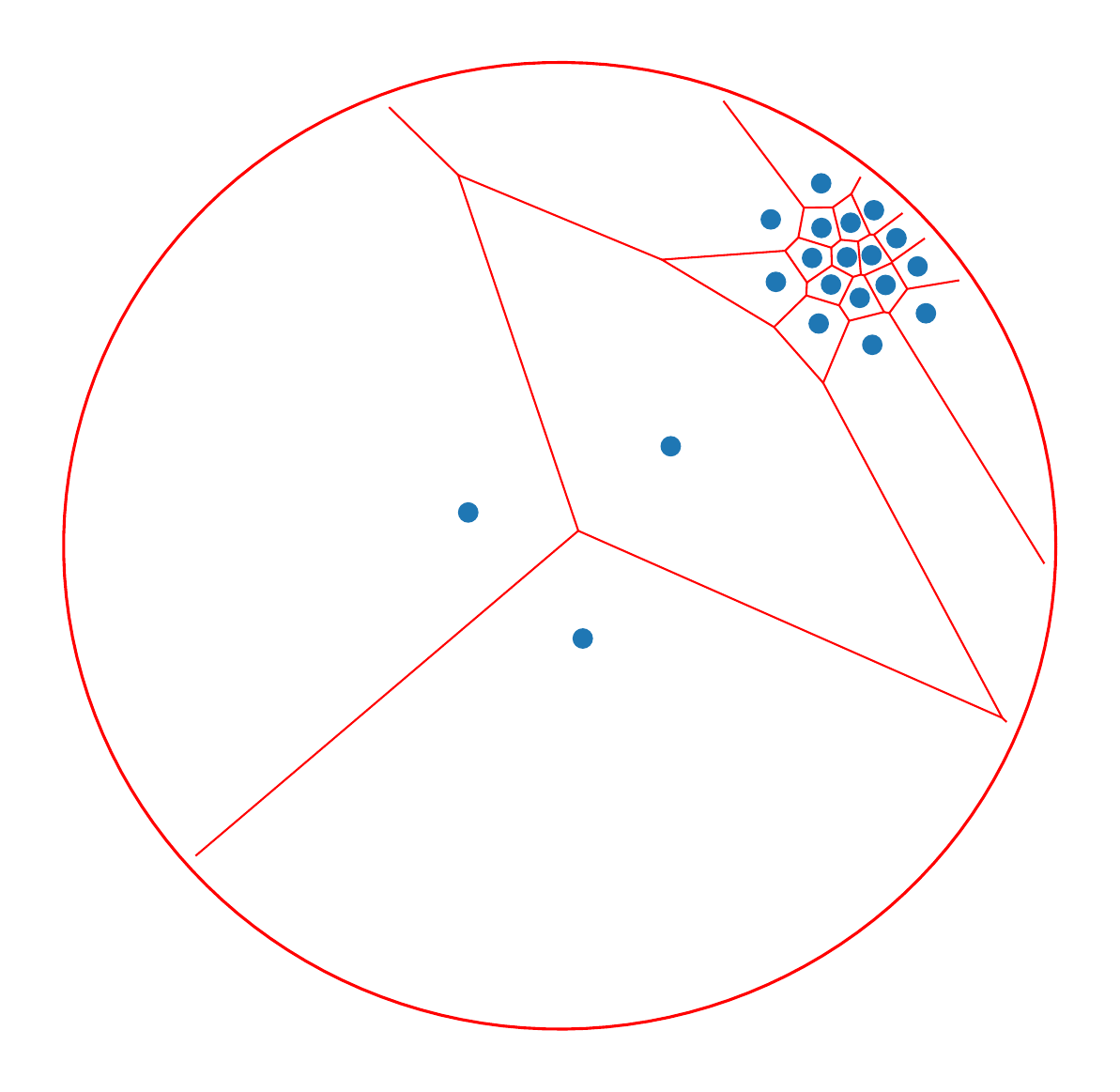} }}
    \caption{(Left) Target Gaussian mixture $\mu$ with two components truncated on a disk. (Middle) Optimal quantization of $\mu$ with $20$ points (blue) and their corresponding  Voronoi cells (in red) after $250$ iterations of Lloyd's algorithm. (Right) Uniform quantization of $\mu$ with $20$ points (blue) and the corresponding power cells (in red) after $5$ iterations of Lloyd's algorithm adjusted for uniform quantization. The diameter of a blue dot is proportional to its weight. The algorithms are randomly initialized and run using the PyMongeAmpere library\textsuperscript{\ref{note1}}.}
    \label{Nuage_Lloyd-mélange1}
\end{figure}

\paragraph{Lloyd-type algorithms.} In his early paper \cite{refLloyd}, Lloyd proposes an algorithm to solve the least squares quantization problem in the univariate setting (without mentioning optimal transport at that time). Lloyd defines a quantization scheme using intervals and their centroids and introduces an iterative algorithm consisting in alternating optimization steps on the intervals and on the centroids. The generalization of this method to the multivariate setting involves \emph{cendroidal Voronoi cells} \cite{Lloydoptcvonedimlogconcave2} which can be interpreted as optimal transport maps in the optimal quantization formulation in \eqref{opt_quant_0}. In this setting, Lloyd's method alternates between evaluation of Voronoi cells (given centroids), and update of centroids (given Voronoi cells). This algorithm corresponds to the K-means algorithm for solving \eqref{opt_quant_0} when the target measure $\mu$ is a point cloud (e.g. an empirical measure of i.i.d observations sampled from an unknown probability distribution).

Lloyd's method can be extended to the case of uniform quantization in \eqref{eq:unif_quantization_0}, for which the Voronoi cells encompass a volume constraint \cite{balzer2009capacity}. We refer to them as \emph{power cells}. These regions are also associated to optimal transport maps in the uniform quantization problem in \eqref{eq:unif_quantization_0} \cite{de2012blue}, and algorithms have been developed for computing these power cells \cite{refnewton}.
Both methods, for optimal or uniform quantization,can be interpreted as a fixed step-size gradient descent, and share the same overall ``alternating'' structure. We will refer to both of them as Lloyd's method. 
They are illustrated in Figure \ref{Nuage_Lloyd-mélange1}, generated using the PyMongeAmpere library\footnote{\label{note1}\href{https://github.com/mrgt/PyMongeAmpere}{https://github.com/mrgt/PyMongeAmpere}}. A precise description of Lloyd's method for optimal and uniform quantization is given in Section \ref{sec:mainResults} with Algorithm \ref{alg:LOYD_optimal_quant} and \ref{alg:LOYD_uniform_quant}. We consider the simplest form of the algorithm, and many variants exist in the literature, including backtracking line search  \cite{de2012blue,refhess} , and a stochastic version based on subsampling, also known as Competitive Learning Vector Quantization (CLVQ) \cite{bouton1997multidimensional}.  Second order Newton's type method where proposed in \cite{refhess} to solve \eqref{eq:unif_quantization_0}. Typical convergence guaranties for such methods are local and require invertibility of the Hessian, an assumption which is not guaranteed in general for \eqref{opt_quant_0} and \eqref{eq:unif_quantization_0}\footnote{For example, if the target measure $\mu$ has a rotation invariance, then the local minima of \eqref{opt_quant_0} or \eqref{eq:unif_quantization_0} are not isolated and the Hessian is not invertible there.}.

We address the question of sequential convergence of the iterates produced by Lloyd's method for solving \eqref{opt_quant_0} and \eqref{eq:unif_quantization_0}. More precisely, we consider the asymptotic stabilization of the support points $y_1,\ldots, y_N$ of the Dirac masses to a critical point of the quantization functional. This question was already considered in previous literature.
For optimal quantization, the first analyses were proposed for the univariate case \cite{Lloydoptimallocalcv,kieffer1982exponential} with special cases such as the strict log-concavity of the target measure density \cite{Lloydoptcvonedimlogconcave1,Lloydoptcvonedimlogconcave2}. In the multivariate setting, it is known that accumulation points of Lloyd's iteration are critical points of the loss in \eqref{opt_quant_0} \cite{LloydCVT,nondegeneracyLloyd,cvlloydanydimpages}, which also holds true for the stochastic variant, CLVQ, under a stepsize decrease conditions \cite{pages1998space}. In particular if such critical points are isolated, an assumption very difficult to check in practice, then the sequence converges. 
For uniform quantization, similar results exist. It was shown for example that accumulation points of Lloyd's method are critical points of the functional in \eqref{eq:unif_quantization_0} \cite{nonassymptotic}. Beyond existing results, a general sequential convergence guaranty for Lloyd's methods in the semi-discrete setting remains open. 

 We provide such guaranties under an explicit definability assumption on the density of the target measure $\mu$. Since the method is the basis for many quantization algorithms, this opens the possibility in future research to adapt our arguments for variants of the Lloyd's algorithm for solving  \eqref{opt_quant_0} and \eqref{eq:unif_quantization_0}.
Note that, we do not consider convergence toward the global minimizer of the quantization functional and limit ourselves to convergence of the sequence to critical points.

\paragraph{Convergence of gradient sequences and rigidity of semi-discrete optimal transport losses.} 
Both results rely on the interpretation of Lloyd's methods as gradient schemes for problems \eqref{opt_quant_0} and \eqref{eq:unif_quantization_0} \cite{Lloydoptcvonedimlogconcave2,nonassymptotic}. In a non-convex setting, the iterates of gradient sequences converge under \L{ojasiewicz} inequality, for example if the underlying loss function is analytic \cite{AbsMahAnd2005}. Kurdyka proposed a generalization of \L{ojasiewicz}'s property that holds true for all functions that are definable in an o-minimal structure \cite{KurdykaarticleominimpliqueKL}, which we refer to as the Kurdyka-\L{ojasiewicz} (KL) inequality. Definable functions represent a broad and versatile class of functions and covers many applications. For example this class is closed under composition, partial minimization / maximization, etc. Convergence of first order methods under KL assumptions have now become standard in non-convex optimization, see for example \cite{attouch2013convergence,bolte2014proximal} and references therein. This is the main path that we follow.

Applying Kurdyka's result to \eqref{opt_quant_0} and \eqref{eq:unif_quantization_0} requires to justify that these functions are definable in some o-minimal structure. This is far from direct as the description of these functionals involve integrals and o-minimal structures are not stable under integration in general. This is one of the great challenges of this field, several partial answers are known, one of them allowing to treat the definability of globally subanalytic integrals \cite{lion1998integration,comte2000nature,Cluckers2009StabilityUI} which have a log-analytic nature. This is the crucial step behind our convergence results, an observation already made in \cite{bolte2023subgradient} in a stochastic optimization context. We work under the assumption that the compactly supported target measure has a globally subanalytic density (see Assumption \ref{ass:targetMeasure}). This allows to justify the resulting log analytic nature of the losses in \eqref{opt_quant_0} and \eqref{eq:unif_quantization_0}. Convergence then follows using the well known connection between Lloyd's method and the gradient algorithm and its convergence analysis under KL assumptions.

The proposed analysis opens the question of definability of semi-discrete optimal transport losses beyond the $2$-Wasserstein distance in \eqref{opt_quant_0} and \eqref{eq:unif_quantization_0}. We justify the definable nature of several such losses involving general optimal transport costs, sliced Wasserstein distance \cite{rabin2012wasserstein,slicedradonbarycenter,bobkov2019one}, the max-sliced Wasserstein distance \cite{kolouri2019generalized,deshpande2019max}, and entropy regularized optimal transport \cite{cuturi2013sinkhorn}. These results are of independent interest, and may be relevant to study numerical applications of optimal transport.

\paragraph{Organization of the paper and notations.}
We start with a description of Lloyd's methods and state our main results in Section \ref{sec:mainResults}. The central convergence arguments based on KL inequality are described in Section \ref{sec:convergence}, and the definability of the underlying losses is described in Section \ref{section_rigidity_of_semi_discrete_losses}. Extension to broader optimal transport losses is given in Section \ref{definability losses}.

Regarding technical elements discussed in the introduction, we recall the definition of KL inequality and the main result of Kurdyka in Section \ref{section_KL}, while the necessary background on o-minimal structures and globally subanalytic sets is given in Section \ref{section_rigidity_of_semi_discrete_losses}. Furthermore, the precise definition of the $W_2$ distance in the semi-discrete setting of \eqref{opt_quant_0} and \eqref{eq:unif_quantization_0} is postponed to Section \ref{extension_general_cost}.

The Euclidean norm and the dot product in $\R^d$ are denoted $\|\cdot\|$ and $\langle\cdot , \cdot\rangle$. We define respectively the open and closed $d$-dimensional Euclidean ball centered in $x$ with radius $r$ as $B_{d}(x,r)$ and $\overline{B_{d}(x,r)}$.

\section{Main results}
\label{sec:mainResults}

We use the notation $Y\in(\R^d)^N$ to denote $N$ points $y_1,\ldots, y_N$ in $\R^d$ which are candidate support points for quantization. This section contains our main results regarding the sequential convergence of Lloyd's algorithms for optimal and uniform quantization.

\subsection{Assumptions on the target measure}
\label{sec:pedagogyOMin}
We will consider the following assumptions on the target measure $\mu$ and its density probability function $f$.
\begin{assumption}
	\label{ass:targetMeasure} 
	$\mu$ is a probability measure on $\R^d$, with compact support, absolutely continuous with respect to Lebesgue measure, with globally subanalytic density $f$.
\end{assumption}
Strictly speaking, the density $f$ consists of an equivalence class of functions defined up to a Lebesgue null set, our assumption means that one element of the class is globally subanalytic, and we will refer to this term as the density $f$.
Absolute continuity and compacity ensure that Lloyd's iterations are well defined, see for example \cite[Assumption 3.1]{nondegeneracyLloyd}.
The global subanaliticity assumption is a rigidity assumption which will allow to invoke KL inequality for our convergence analysis. 
The precise definition of global subanaliticity is given in Section \ref{sec:oMin}. This is a very versatile notion for which we provide a few examples below.
\begin{exm}\label{exmp_GSA}
The following are globally subanalytic density functions:
\begin{itemize}
    \itemsep0em 
	\item The uniform density on a compact basic semi-algebraic set, of the form $S  = \{x,\, P_i(x) \geq 0, i =1,\ldots,m\} \subset \R^d$, for some $m \in \N$, where $P_i \colon \R^d \to \R$ are polynomials. 
	\item Densities of the form $x \mapsto f(x) \mathbbm{1}_{S}(x)$ where $S$ is a compact semi-algebraic set (as above) and $f\colon U \to \R $ is an analytic function on an open domain $U \supset S$ (Lemma \ref{prop_exmp_GSA}). These include truncated Gaussians for example. 
	\item Semi-algebraic densities including integrable rational functions and their restrictions to compact sets.
    \item Bounded subanalytic functions.
	\item Any finite mixture of the above density examples. This allows to consider nonsmooth and possibly discontinuous densities.
    \item For instance, any normalized image considered as a piecewise constant probability density on the plane is globally subanalytic.
\end{itemize}
\end{exm}

\begin{rem}
	Under Assumption \ref{ass:targetMeasure}, the density $f$ is strictly positive on a full measure open dense set in $\supp\,\mu$, the support of the measure $\mu$. Indeed, the set $\{x,\:f(x)>0\}$ is globally subanalytic and can be partitioned into a finite number of embedded submanifolds in $\R^d$ \cite[4.8]{driesmiller}. It is enough to keep those of dimension $d$ only since those of dimension at most $d-1$ have zero Lebesgue measure. Additionally, the submanifolds $(C_i)_{i\in I}$ of dimension $d$ are precisely open subsets of $\R^d$. Therefore we can consider the set $S :=\cup_{i\in I}C_i$ which is is non-empty, globally subanalytic and open. We have $\supp \mu = \mathrm{cl}_{\R^d}(S)$, since the support is the smallest closed set of nonzero measure, and $\dim(\supp\mu \backslash S) < d$, using \cite[4.7]{driesmiller}, which is the claimed statement.
	\label{rem:densityPositive}
\end{rem}

\subsection{Optimal Quantization}
\label{section_optimal_quantization}
In addition to the main hypothesis above, the target measure is assumed to have convex support throughout this section.
\begin{assumption}
	$\mu$ is a probability measure as in Assumption \ref{ass:targetMeasure}, with convex support.
	\label{ass:convexSupport}
\end{assumption}
Partial minimization of the objective function in \eqref{opt_quant_0} with respect to the weights $\pi$ leads to the \emph{optimal quantization functional} $G_N \colon (\R^d)^N \to \R$, see e.g. \cite{Bookquantization},
\begin{equation}
    \label{eq:optimal_quant}
	G_N(Y) \quad = \quad  \underset{\pi \in \Delta_{N}}{\min} \quad \frac{1}{2}W_{2}^{2}\left(\mu,\sum_{i=1}^N \pi_i\delta_{y_i}\right)\quad=\quad \frac{1}{2} \int \min_{i=1,\ldots,N} \|x - y_i\|^2 f(x)dx,
\end{equation}
where we recall that $f$ stands for the density of $\mu$. Problem \eqref{opt_quant_0} therefore amounts to minimize the objective $G_N$. 
This functional is tightly connected to the so-called Voronoi cells $(V_i(Y))_{i=1,\ldots,N}$ associated to the support points $Y = (y_i)_{i=1}^N \in (\mathbb{R}^{d})^N$ \cite[Proposition 3.1]{Lloydoptcvonedimlogconcave2}, and given for $i=1,\ldots,N$ by 
\begin{equation}\label{def:Voronoi_cells}
    V_{i}(Y)=\left\{x\in \mathbb{R}^{d}\:|\: \forall j=1,\ldots,N \: :\:\|x-y_{i}\|^{2} < \|x-y_{j}\|^{2}\right\}.
\end{equation}
The Voronoi cells in \eqref{def:Voronoi_cells} are open and disjoint, they are all non-empty as long as $Y$ does not belong to the generalized diagonal defined by
\begin{align}
	\label{eq:generalizedDiagonal}
	D_{N}&=\{ Y\in (\mathbb{R}^{d})^{N}\: | \: \exists i \ne j: y_{i} = y_{j}\}.
\end{align}
Therefore, for any $Y \not \in D_N$, the optimal quantization functional in \eqref{eq:optimal_quant} takes the form
\begin{equation}
    \label{eq:optimal_quant2}
    G_N(Y) \quad=\quad \frac{1}{2}\sum_{i=1}^{N}\int_{V_{i}(Y)}\|x-y_{i}\|^{2}f(x)dx.
\end{equation}
Note that a local minimizer to \eqref{opt_quant_0} on $(\supp\mu)^N$ does not lie in $D_N$ \cite[Proposition 3.5]{Lloydoptcvonedimlogconcave2}.
Lloyd's algorithm for computing optimal quantizers is then defined by fixed point iterations of the Voronoi barycentric mapping $T_N \colon (\supp\mu)^N \setminus D_N \to (\supp\mu)^N \setminus D_N$ defined by
\begin{align}\label{Lloyd_map}
    T_N(Y) = \left( \frac{\int_{V_{i}(Y)}xd\mu(x)}{\mu(V_{i}(Y))} \right)_{i=1,\ldots,N}, 
\end{align}
Then the map $T_N$ is indeed well defined on $(\supp\mu)^N \setminus D_N$ since, for distinct centroids in $(\supp \mu)^N$, the Voronoi cells in \eqref{def:Voronoi_cells} are non-empty and have non-zero $\mu$ measure. Furthermore $T_N$ takes values in $ (\supp \mu)^N \setminus D_N$ as the centroids remain distinct since the Voronoi cells are disjoint, and they belong to the support of $\mu$ since it is convex by Assumption \ref{ass:convexSupport}.
This is summarized in Algorithm \ref{alg:LOYD_optimal_quant} and we prove in Section \ref{sec:optimal_quantization} the following theorem regarding convergence of the iterates.

\begin{thm}\label{th:main_Optimal_Quant}
	Let $\mu$ be as in Assumption \ref{ass:convexSupport}, then the iterates of Lloyd's algorithm for optimal quantization in Algorithm \ref{alg:LOYD_optimal_quant} converge to a critical point of $G_N$ in \eqref{eq:optimal_quant}.
\end{thm}

\begin{rem}
Albeit similar to the K-means algorithm, our proof techniques for the convergence of the iterates of Lloyd's algorithm are not applicable to the K-means functional because of its non-differentiability. More precisely, in our semi-discrete setting \eqref{eq:optimal_quant2} corresponding to the optimal quantization of a continuous measure $\mu$, the functional $G_N$ writes as a sum of integrals over power cells. Although the cells evolve with the points cloud $Y$, we may use a modified form of the theorem of integration under the integral sign (or Reynold's transport theorem) to prove differentiability. However, for the K-means functional, that is $G_N$ in \eqref{eq:optimal_quant} when the target measure $\mu$ is discretely supported, one needs to consider subtle non-differentiability issues. General considerations on the convergence of the K-means algorithm are given in \cite{bottou1994convergence}.
\end{rem}

\begin{algorithm}[t]
  \caption{Lloyd's algorithm for optimal quantization}
	\label{alg:LOYD_optimal_quant}
	\begin{algorithmic}
	\STATE \textbf{Input:} $\mu$ absolutely continuous probability measure on $\R^d$ with convex compact support, $Y_0  \in (\supp\mu)^N \setminus D_N$ as in \eqref{eq:generalizedDiagonal}.
    \FOR{$n \in \N$}
        \STATE $Y_{n+1} = T_N(Y_n)$ where $T_N$ is the Voronoi barycentric mapping in \eqref{Lloyd_map}.
    \ENDFOR
 	\end{algorithmic}
\end{algorithm}

\subsection{Uniform Quantization}\label{section_Uniform_Quantization}

The objective in \eqref{eq:unif_quantization_0} is called the \emph{uniform quantization functional} $F_N \colon  (\mathbb{R}^{d})^{N} \to \R$, and is defined as follows
\begin{equation}\label{eq:unif_quantization}
	F_N(Y) \quad = \quad  \frac{1}{2}W_{2}^{2}\left(\mu,\frac{1}{N}\sum_{i=1}^{N} \delta_{y_{i}}\right).
\end{equation}
The uniform quantization problem consists in finding minimizers of $F_N$. Similar to optimal quantization, the functional can be explicitly described by power cells (also called Laguerre cells \cite{nonassymptotic}). These generalize Voronoi cells and characterize the optimal transport plan in $W_2$ between the target measure and the discrete measure (see also  \eqref{eq:generalCostOptTransport}). 
If $Y$ is not in the generalized diagonal $D_{N}$, then $F_N$ in \eqref{eq:unif_quantization} can be written with the Kantorovich dual formulation introduced as follows for $w\in\R^N$ (see for example \cite[Theorem 5.9]{villanioptoldnew} and \cite[Section 2]{nonassymptotic} for details on the uniform quantization setting)
\begin{equation}
	\label{Objective_G}
	\begin{split}
		\Phi(Y,w)&:=\frac{1}{2}\left\{\int_{\mathbb{R}^{d}}{\left(\min_{i=1,\ldots,N}\|x-y_{i} \|^{2}-w_{i}\right)f(x)dx}+\frac{1}{N}\sum_{i=1}^{N}{w_{i}}\right\}\\
		&=\frac{1}{2}\left\{\sum_{i=1}^{N}\int_{\Pow_{i}(Y,w)}\left(\|x-y_{i} \|^{2}-w_{i}\right)f(x)dx+\frac{1}{N}\sum_{i=1}^{N}{w_{i}}\right\}
	\end{split}
\end{equation}
where the power cells $\Pow_{i}(Y,w)$ are defined for $i=1,\ldots,N$ as
\begin{equation}
	\label{power_cells}
	\Pow_{i}(Y,w)=\{ x\in \mathbb{R}^{d}\: |\: \forall j=1,\ldots,N \: :\: \|x-y_{i}\|^{2}-w_{i} \leq \|x-y_{j}\|^{2}-w_{j} \}.
\end{equation}
Kantorovich's duality states that for any $Y \not \in D_N$, the uniform quantization functional in \eqref{eq:unif_quantization} is given by
\begin{equation}
	\label{Objective_function}
	F_{N}(Y) = \max_{w\in \mathbb{R}^{N}}\ \Phi(Y,w).
\end{equation}
It can be shown (see Remark \ref{rem_unicity}) that the maximum in \eqref{Objective_function} is attained and unique up to the addition of a constant. Therefore, it defines a unique set of power cells \eqref{power_cells}, and additionally the $\mu$ measure of each power cell is $\frac{1}{N}$.

We introduce $B_N \colon (\R^d)^N \setminus D_N \to  (\R^d)^N \setminus D_N$ the Laguerre barycentric application defined by
\begin{equation}
	\label{def:B_N}
	B_{N}(Y)=\left(N\int_{\Pow_{i}(Y,\phi(Y))}xd\mu(x)\right)_{i=1,\ldots,N}
\end{equation}
where $\phi(Y)$ denotes an element of ${\argmax}_{w\in \mathbb{R}^{N}}\:\Phi(Y,w)$ in \eqref{Objective_G}. Note that the Laguerre cells are disjoint and non-empty outside of the generalized diagonal $D_N$ so that $B_N$ indeed has values in $(\R^d)^N \setminus D_N$.
Lloyd's algorithm consists in fixed point iterations of the map $B_N$ as summarized in Algorithm \ref{alg:LOYD_uniform_quant}. 
Our main result regarding this algorithm is the following, whose proof is postponed to Section \ref{sec:uniform_quant}.

\begin{thm}\label{th:main}
	Let $\mu$ be as in Assumption \ref{ass:targetMeasure}, then the iterates of Lloyd's algorithm for uniform quantization in Algorithm \ref{alg:LOYD_uniform_quant} converge to a critical point of $F_N$ in \eqref{eq:unif_quantization}.
\end{thm}

\begin{algorithm}[t]
  \caption{Lloyd's algorithm for uniform quantization}
	\label{alg:LOYD_uniform_quant}
	\begin{algorithmic}
		\STATE \textbf{Input:} $\mu$ absolutely continuous probability measure on $\R^d$ with compact support, \\
		$Y_0  \in (\R^d)^N \setminus D_N$ as in \eqref{eq:generalizedDiagonal}.
	    \FOR{$n \in \N$}
	        \STATE $Y_{n+1} = B_N(Y_n)$ where $B_N$ is the Laguerre barycentric mapping in \eqref{def:B_N}.
	    \ENDFOR
 	\end{algorithmic}
\end{algorithm}
Our result does not provide details regarding the nature of the critical points of $F_N$ found by Algorithm \ref{alg:LOYD_uniform_quant}. Numerical experiments suggest that random initializations $Y_0\in(\R^d)^N$ yield different functional values at convergence, implying that there is no general convergence to a global minimum. On the other hand, as Corollary 4 in \cite{nonassymptotic} shows, if the initial points are sufficiently spaced out, one step of Lloyd's algorithm is enough to obtain a low value of $F_N$, suggesting that the algorithm finds critical points with low objective values.

\begin{rem}
	Contrary to optimal quantization, the convexity of the support is not required here for the algorithm to be well defined, by construction of the power cells at optimality in \eqref{Objective_function}. Additionally, the identity in \eqref{Objective_function} only holds if $Y \not \in D_N$. It is nonetheless possible to extend it to equality cases, that is when there exists $i,j\in\{1,\ldots,N\}, i\neq j$, such that $y_i=y_j$, and this is done in Theorem \ref{th:KantorovichDiagonal}. 
	This remark does not have any influence on our algorithmic results since the iterates remain away from the diagonal, but it will be essential to justify that the function $F_N$ in \eqref{eq:unif_quantization} is definable on the whole space.
	\label{rem:unifQuantEverywhere}
\end{rem}

\subsection{Rigidity of quantization functionals and semi-discrete optimal transport losses}

The key in obtaining our convergence results is to interpret Lloyd's algorithms as gradient descent and rely on convergence analysis under Kurdyka-\L{ojasiewicz} assumptions \cite{AbsMahAnd2005,bolte2023subgradient}. A sufficient condition for applying this analysis is that the associated loss functions are definable in an o-minimal structure \cite{KurdykaarticleominimpliqueKL}.
While definability is stable under many operations, such as composition, inverse, and projection, it is not a general rule that infinite summations or integrals preserve definability. For the specific case of globally subanalytic integrands, the resulting partial integrals have a log-analytic nature \cite{lion1998integration,comte2000nature,Cluckers2009StabilityUI,kaiser2013integration}, so that such parameterized integrals are definable in the o-minimal structure  $\mathbb{R}_{\text{an,exp}}$ \cite{Dries1995OnTR}. This was also observed in the context of stochastic optimization in \cite{bolte2023subgradient}. We prove in Section  \ref{definability losses} that several loss functions based on semi-discrete optimal transport divergences are definable in this o-minimal structure under Assumption \ref{ass:targetMeasure}, including the optimal quantization functional in \eqref{eq:optimal_quant} and the uniform quantization functional in \eqref{eq:unif_quantization}. This definability result holds for optimal transport with general globally subanalytic cost functions (Section 6 in \cite{villanioptoldnew}), the max-sliced Wasserstein distance \cite{kolouri2019generalized} and the entropy regularized Wasserstein divergence \cite{Entropic}. These constitute technical extensions of our main sequential convergence analysis which are of independent interest.

\section{Convergence of the iterates of Lloyd's algorithms}
\label{sec:convergence}
This section provides the proof arguments for Theorems \ref{th:main_Optimal_Quant} and \ref{th:main}. The main device is the interpretation of Lloyd's algorithms as gradient methods combined with Kurdyka-\L{ojasiewicz} (KL) inequality. This allows to invoke sequential convergence results for gradient methods on functions verifying KL inequality \cite{AbsMahAnd2005,attouch2013convergence}. The KL inequality is verified in our cases by Kurdyka's sufficient condition, which is the definability of the quantization loss functions in \eqref{eq:optimal_quant} and \eqref{Objective_function}. The proof of these definability results is independent of the convergence analysis and is given in Section \ref{sec:definability_quantization}.

\subsection{KL inequality}\label{section_KL}
\L{ojasiewicz}'s gradient inequality is of the form $\|f(x)-f(y)\|^{\theta} \leq C\|\nabla f(x)\| $ with $\theta \in ]0,1[$ for a fixed $y$, and was first identified for real analytic functions \cite{Lojasiewicz1963propriete}. Later on, Kurdyka \cite{KurdykaarticleominimpliqueKL} proposed a generalization to functions definable in an o-minimal structure now known as the Kurdyka-\L{ojasiewicz} (KL) inequality. The KL inequality usually takes the following form.
\begin{defn}[Definition 7 in \cite{proximalbolte2010defKL}]\label{def:KL_property}
    Let $V$ be an open set of $\mathbb{R}^{d}$ and $\bar{Y}$ be a point in $V$. We say that a $\mathcal{C}^{1}$ function $F: V \to \mathbb{R}$ has the KL property at $\bar{Y}$ if there exist $\eta\in \ ]0,+\infty[$, a neighborhood $U$ of $\bar{Y}$ in $V$ and a continuous positive concave function $\Psi:[0,\eta[ \ \to \mathbb{R}_{+}$  such that:
    \begin{enumerate}
        \itemsep0em 
        \item $\Psi(0)=0$.
        \item $\Psi\in \mathcal{C}^{1}(]0,\eta[)$.
        \item $\forall t \in \ ]0,\eta[,\: \Psi'(t)>0$.
        \item $\forall \ Y\in U\cap \{Y\in \mathbb{R}^{d}\:|\: F(\bar{Y})< F(Y) < F(\bar{Y})+\eta\}$ the KL inequality holds:
            \begin{equation*}
            \Psi'(F(Y)-F(\bar{Y}))\ \|\nabla F(Y)\|\geq 1.
            \end{equation*}
    \end{enumerate}
     If $F$ has the KL property at all $ \bar{Y} \in V$, we say that $F$ has the KL property on $V$.
\end{defn}

\begin{thm}[Theorem 1 in \cite{KurdykaarticleominimpliqueKL}]
\label{th:kurdyka}
Let $V$ be an open set on $\R^d$ and $f \colon V \to \R$ be $\mathcal{C}^1$ and definable in an o-minimal structure. Then $f$ satisfies the KL property, as described in Definition \ref{def:KL_property}.
\end{thm}

We postpone a precise definition and detailed discussion of definability of semi-discrete optimal transport losses to Section \ref{section_rigidity_of_semi_discrete_losses}. In particular, we will show that the objective functions $G_{N}$ and $F_{N}$ defined in \eqref{eq:optimal_quant} and \eqref{Objective_function} are definable in some o-minimal structure, under Assumption \ref{ass:targetMeasure}.

As a consequence of Theorem \ref{th:kurdyka}, in the following we may use the KL inequality for $G_{N}$ and $F_{N}$ to prove the convergence of Lloyd's iterates. Lloyd's algorithms can indeed be interpreted as gradient methods (see \eqref{eq:grad_GN} and \eqref{eq:grad_FN}) for which convergence analysis under KL assumptions has become standard \cite{AbsMahAnd2005,proximalbolte2010defKL,attouch2013convergence,bolte2014proximal}.

\subsection{Lloyd's sequence for optimal quantization}
\label{sec:optimal_quantization}
Lloyd's method for optimal quantization, given in Algorithm \ref{alg:LOYD_optimal_quant}, can be understood as a gradient descent algorithm on the optimal quantization objective function $G_N$ in \eqref{eq:optimal_quant}. Throughout this section, we interpret $Y\in (\mathbb{R}^{d})^{N}$ as a matrix in $\R^{N \times d}$.
The function $G_N$ is differentiable outside of the generalized diagonal, and its gradient is explicit, as stated in the following proposition.
\begin{prop}[Proposition 6.2 in \cite{Lloydoptcvonedimlogconcave2}]
	Let $\mu$ be as in Assumption \ref{ass:convexSupport} and $Y\in (\mathbb{R}^{d})^{N}\setminus D_{N}$. Then $G_{N}$ is differentiable in $Y$ with gradient
    \begin{equation}
        \nabla G_{N}(Y)=M(V(Y))\ (Y-T_N(Y)),
        \label{eq:grad_GN}
    \end{equation}
	where $M(V(Y)) \in \R^{N \times N}$ is the diagonal matrix with entries $(\mu(V_i(Y)))_{i=1,\ldots,N}$, with $(V_i(Y))_{i=1,\ldots,N}$ the Voronoi cells in \eqref{def:Voronoi_cells}, and $T_N$ is the Voronoi barycentric mapping defined in \eqref{Lloyd_map}.
	\label{prop:lloydOptimalGradient}
\end{prop}
We denote $(Y_n)_{n\geq 0}$ the Lloyd sequence in Algorithm \ref{alg:LOYD_optimal_quant}. Since $Y_n \in (\supp \mu)^N \setminus D_N$ for all $n$, the corresponding Voronoi cells are non empty and the diagonal matrix $M(V(Y_n))$ in Proposition \ref{prop:lloydOptimalGradient} is invertible for all $n \in \N$. Therefore, $Y_n$ can be rewritten for all $n\geq 1$ in a pre-conditioned gradient descent formulation as 
\begin{equation}\label{LloydCVT_algo_2}
	Y_{n+1}=Y_{n}-M(V(Y_{n}))^{-1}\  \nabla G_{N}(Y_{n}).
\end{equation}
It turns out that the diagonal matrix $M(V(Y_n))$ remains bounded away from $0$ for all $n \in \N$ as a consequence of \cite[Theorem 3.6]{nondegeneracyLloyd} since Assumption \ref{ass:targetMeasure} ensures that \cite[Assumption 3.1]{nondegeneracyLloyd} holds; see Remark \ref{rem:densityPositive}.
\begin{prop}[Corollary 3.7 in \cite{nondegeneracyLloyd}]
	Let $\mu$ be as in Assumption \ref{ass:convexSupport} and $(Y_{n})_{n\geq 0}$ be the Lloyd sequence defined in Algorithm \ref{alg:LOYD_optimal_quant}, then there exists $\ell>0$ such that for all $i=1,\ldots,N$
	\begin{equation}
	    \label{eq:mu_bounded}
	    \inf_{n\geq 0} \mu(V_{i}(Y_{n}))\geq \ell.
	\end{equation}	
	\label{prop:muBounded}
\end{prop}
As a consequence, we get the following strong descent condition.
\begin{lem}\label{SGD_LloydCVT}
	Let $\mu$ be as in Assumption \ref{ass:convexSupport} and $(Y_{n})_{n\geq 0}$ be the Lloyd sequence defined Algorithm \ref{alg:LOYD_optimal_quant}. Then, given $\ell$ as in \eqref{eq:mu_bounded}, for all $n \in \N$, 
	\begin{align}
		&G_N(Y_n)-G_N(Y_{n+1})\geq \frac{\ell}{2} \|\nabla G_N(Y_{n})\|  \ \|Y_{n+1}-Y_{n}\|. \label{SGD1} \tag{SDC}
	\end{align}
\end{lem}

\begin{proof}
	Recall that $Y_n \not \in D_N$ for all $n \in \N$.
	Let us define for all $Y,Z\in(\R^d)^N \setminus D_N$ and for $i=1,\ldots,N$ the function $\mathcal{H}_{i}(Y,Z)=\int_{V_{i}(Y)}\|x-z_{i}\|^{2}f(x)dx$. Then for all $i=1, \ldots, N$, we have
	\begin{align*}
            \mathcal{H}_{i}(Y_{n},Y_{n})&=\int_{V_{i}(Y_{n})}\|x-y^{n}_{i}\|^{2}f(x)dx\\
            &=\int_{V_{i}(Y_{n})}\|x-y_{i}^{n+1}\|^{2}f(x)dx + \int_{V_{i}(Y_{n})}\|y_{i}^{n+1}-y_{i}^{n}\|^{2}f(x)dx\\
            &=\mathcal{H}_{i}(Y_{n},Y_{n+1})+\mu(V_{i}(Y_{n}))\|y_{i}^{n+1}-y_{i}^{n}\|^{2}.
    \end{align*}
where the second equality is verified since by construction of the sequence $(Y_{n})_{n\geq 0}$ in \eqref{LloydCVT_algo_2} , we have $y_i^{n+1}=\frac{1}{\mu(V_i(Y_n))}\int_{V_i(Y_n)}xd\mu(x)$ and thus $\int_{V_{i}(Y_{n})}\langle x-y_{i}^{n+1},y_{i}^{n+1}-y_{i}^{n} \rangle d\mu(x)=0$ .  
	Additionally, denoting $\HH(Y,Z) = \sum_{i=1}^N \HH_i(Y,Z)$, we get by Lemma 2.1 of \cite{LloydCVT} and using \eqref{eq:optimal_quant2}, that for all $n \in \N$,
	\begin{align*}
		\min_{Z\in(\R^d)^N}\ \HH(Z,Y_{n})&=\HH(Y_{n},Y_{n}) = 2G_N(Y_{n}). 
	\end{align*}
	We deduce that
    \begin{align}
		G_{N}(Y_{n})& = \frac{1}{2} \sum_{i=1}^N \mathcal{H}_{i}(Y_{n},Y_{n})\geq G_{N}(Y_{n+1})+ \frac{1}{2}\sum_{i=1}^{N}\mu(V_{i}(Y_{n}))\|y_{i}^{n+1}-y_{i}^{n}\|^{2}\label{SGD_LLoyd2}\\
        &\geq G_{N}(Y_{n+1})+\frac{1}{2}\sum_{i=1}^{N}\mu(V_{i}(Y_{n}))^{2}\|y_{i}^{n+1}-y_{i}^{n}\|^{2}\nonumber\\
        &=G_{N}(Y_{n+1})+\frac{1}{2}\|\nabla G_{N}(Y_{n})\|\left(\sum_{i=1}^{N}\mu(V_{i}(Y_{n}))^{2}\|y_{i}^{n+1}-y_{i}^{n}\|^{2}\right)^{1/2}\nonumber\\
        &\geq G_{N}(Y_{n+1})+\frac{1}{2}\|\nabla G_{N}(Y_{n})\|\ \min_{i=1,\ldots,N}\mu(V_{i}(Y_{n}))\|Y_{n+1}-Y_{n}\|\nonumber. 
    \end{align}
	From this we get condition \eqref{SGD1} using Proposition \ref{prop:muBounded}.
\end{proof}
This strong descent condition allows to prove sequential convergence provided that $G_N$ is a KL function \cite{AbsMahAnd2005}. This is stated in the following lemma whose proof rely on definability results postponed to Section \ref{section_rigidity_of_semi_discrete_losses}.

\begin{lem}\label{lem_mu_GSA_implies_G_KL}
	Let $\mu$ be as in Assumption \ref{ass:convexSupport}, then the optimal quantization functional $G_{N}$ defined in \eqref{eq:optimal_quant} is a KL function on $(\R^d)^N \setminus D_N$.
\end{lem}
\begin{proof}
	This result is a consequence of Lemma \ref{lem_G_definable} and Theorem \ref{th:kurdyka}.
\end{proof}

The standard analysis of gradient method under KL assumption allows to conclude about convergence of Lloyd's iterates for optimal quantization, the proof of which is given below.

\begin{proof}[Proof of Theorem \ref{th:main_Optimal_Quant}]
	From Lemma \ref{lem_mu_GSA_implies_G_KL}, we have that $G_{N}$ is a KL function. We can then apply \cite[Theorem 3.4]{AbsMahAnd2005}, which requires the strong descent condition \eqref{SGD1} and the property that $G_N(Y_{n+1}) = G_N(Y_n)$ implies $Y_{n+1} = Y_n$, a property satisfied in our case due to the descent condition and the expression in \eqref{LloydCVT_algo_2}. The convergence of the iterates then follows. Note that \cite[Theorem 3.4]{AbsMahAnd2005} is stated without a domain. However the main proof mechanism is a trap argument which is purely local and thus can easily be extended to a closed set containing all the iterates in an open domain. In our setting we work on $(\R^d)^N \setminus D_N$, and Proposition \ref{prop:muBounded} ensures that the sequence $(Y_n)_{n\geq 0}$ in Algorithm \ref{alg:LOYD_optimal_quant} remains away from the generalized diagonal $D_N$ so that the trap argument applies. See also \cite{attouch2013convergence} for a general account using KL functions. 
\end{proof}

\subsection{Lloyd's sequence for uniform quantization}
\label{sec:uniform_quant}
The objective function $F_N$ in \eqref{Objective_function} is differentiable outside the generalized diagonal, and the expression of its gradient is given in the following proposition.
\begin{prop}[Proposition 1 in \cite{nonassymptotic}]
\label{Gradient}
	Let $\mu$ be as in Assumption \ref{ass:targetMeasure}, then for all $Y\in(\R^d)^N\setminus D_{N}$, we have
    \begin{equation}
        \nabla F_{N}(Y)=\frac{1}{N}\left(Y-B_{N}(Y)\right),
        \label{eq:grad_FN}
    \end{equation}
	where $B_N$ is the Laguerre barycentric mapping described in \eqref{def:B_N}.
\end{prop}

\begin{rem}\label{rem_unicity}
The vector $\phi(Y)\in \argmax_{w\in \mathbb{R}^{N}} \ \Phi(Y,w)$ appearing in $F_N$ in the definition of \eqref{Objective_function} is unique up to the addition of a constant (see e.g. Theorem 1.17 in \cite{booksantambrogio}). In order to have uniqueness, it is sufficient to add a constraint on the set of maximizers, such as setting the mean of the components of $w$ to be 0 as done in \cite{refnewton}. Furthermore, all elements in the argmax describe the same power cells.
\end{rem}
As in Section \ref{sec:optimal_quantization}, Proposition \ref{Gradient} allows to interpret the Lloyd sequence for uniform quantization in Algorithm \ref{alg:LOYD_uniform_quant}
as gradient iterations on the uniform quantization functional $F_N$ in \eqref{Objective_function}.
Similarly it is possible to show that $F_N$ restricted to $(\R^d)^N\setminus D_N$ is a KL function, based on its definability, whose proof is postponed to Section \ref{section_rigidity_of_semi_discrete_losses}.
\begin{lem}\label{lem_mu_GSA_implies_F_KL}
	Let $\mu$ be as in Assumption \ref{ass:targetMeasure}, then the objective function $F_{N}$ defined in \eqref{Objective_function} is a KL function on $(\R^d)^N \setminus D_N$.
\end{lem}

\begin{proof}
	It is a direct consequence of Lemma \ref{lem_F_definable} and Theorem \ref{th:kurdyka}.
\end{proof}
We follow the same line as in the optimal quantization setting of Section \ref{sec:optimal_quantization}.

\begin{proof}[Proof of Theorem \ref{th:main}]
	The iterates remain in $(\R^d)^N \setminus D_N$ so that by Proposition \ref{Gradient} we have $Y_{n+1} = Y_n - N \nabla F_N(Y_N)$ for all $n \in \N$.
	From the proof of Proposition 2 in \cite{nonassymptotic}, we have the following inequality
	\begin{equation*}
	    F_{N}(Y_{n})-F_{N}(Y_{n+1})\geq \frac{1}{2N}\|Y_{n+1}-Y_{n}\|^{2},
	\end{equation*}
	for all $n \in \N$, which is a strong descent condition and also entails that $F_{N}(Y_{n+1}) = F_{N}(Y_{n})$ implies $Y_{n+1} = Y_n$.
	Finally Lemma \ref{lem_mu_GSA_implies_F_KL} ensures that $F_{N}$ is a KL function. The convergence follows from \cite[Theorem 3.4]{AbsMahAnd2005}, see also \cite{attouch2013convergence} for general KL functions. Note that similarly as in the proof of Theorem \ref{th:main_Optimal_Quant}, the trap argument can be extended to a compact subset containing all the iterates in $(\R^d)^N \setminus D_N$ as described in the proof of Proposition 2 in \cite{nonassymptotic}.
\end{proof}

\begin{rem}
	A possible approach to obtain asymptotic convergence rates for Lloyd's methods would be to prove a \L{}ojasiwicz inequality for $G_N$ and $F_N$, in order to then apply \cite[Theorem 11]{proximalbolte2010defKL}. The \L{}ojasiwicz inequality corresponds to the KL property in Definition \eqref{def:KL_property} for which the desingularizing function $\Psi$ is a power function. It is well-known in subanalytic geometry that real analytic, subanalytic or semi algebraic functions satisfy the \L{}ojasiwicz inequality, unfortunately the specific form of $G_N$ and $F_N$ does not allow for a precise characterization of its desingularizing function. A more detailed discussion on the topic can be found in Section \ref{subsec:rate_Lloyd}.
\end{rem}

\section{Rigidity of semi-discrete optimal transport quantization losses}
\label{section_rigidity_of_semi_discrete_losses}

In this section we show the definability in an o-minimal structure of the optimal and uniform quantization functionals  $G_{N}$ and $F_{N}$ defined respectively in \eqref{eq:optimal_quant} and \eqref{Objective_function}. This finishes the proof arguments for our main convergence results for Lloyd's algorithms in Theorem \ref{th:main_Optimal_Quant} and \ref{th:main} since by Kurdyka's result recalled in Theorem \ref{th:kurdyka}, they are therefore KL functions. 
For the sake of completeness, we also provide elementary proofs of well-known results in definable geometry. We start with the  introduction of o-minimal structures and recall the main results that will be needed further in this section.

\subsection{O-minimal structures}
\label{sec:oMin}
O-minimal structures describe families of subsets of Euclidean spaces which preserve the favorable rigidity properties of semi-algebraic sets. 
The description is axiomatic. We refer the reader to the lecture notes from Coste \cite{coste2000introduction} for an introduction, to \cite{driesmiller} for an extensive presentation of consequences of o-minimality and to \cite{loi2010lecture} for additional bibliographic pointers. We limit ourselves to structures expanding the field of real numbers.
For statements involving definable objects, if not precised otherwise, the word ``\emph{definable}'' implicitly means that all object are definable in the same o-minimal structure.

\begin{defn}
	Let $\mathcal{S}=(\mathcal{S}_{n})_{n\in \mathbb{N}}$ be such that for each $n \in \N$, $\mathcal{S}_{n}$ is a family of subsets of $\mathbb{R}^{n}$. It is called an \emph{o-minimal structure} if it satisfies the following:
	\begin{itemize}
    \itemsep0em 
		\item for each $n \in \N$, $\mathcal{S}_n$ contains all algebraic subsets (defined by polynomial equalities).
		\item for each $n \in \N$, $\mathcal{S}_n$ is a boolean sub-algebra of $\R^n$: stable under intersection, union and complement.
		\item $\mathcal{S}$ is stable under Cartesian product and projection on lower dimensional subspace.
		\item $\mathcal{S}_{1}$ is exactly the set of finite unions of points and intervals.
	\end{itemize}
	 A set belonging to an o-minimal structure is said to be definable (in this structure). A function whose graph or epigraph is definable is also called definable (in this structure).
\end{defn}

\subsubsection{Semi-algebraic subsets}
Semi-algebraic sets represent the smallest o-minimal structure. In other words, any o-minimal structure contains all the semi-algebraic sets.

\begin{defn}\label{semi_algebraic}
A set $S$ is a basic semi-algebraic subset of $\mathbb{R}^{n}$ if there exists a finite number of real polynomial functions $(f_{i})_{i=1,\ldots,N}$ and $(g_{j})_{j=1,\ldots,M}$ such that:

\begin{equation*}
    S=\{x\in \mathbb{R}^{n}\:|\:f_{1}(x)> 0,\:\ldots\:,f_{N}(x)>0,\: g_{1}(x)=0,\:\ldots\:,g_{M}(x)=0\}.
\end{equation*}
A subset which is a finite union of basic semi-algebraic subsets is semi-algebraic.
\end{defn}
The following fact is a consequence of Tarski-Seidenberg theorem, see for example \cite{coste2000introduction}.
\begin{prop}
    The collection of all semi-algebraic subsets, denoted $\mathcal{SA}$, forms an o-minimal structure. 
\end{prop}

\subsubsection{Globally subanalytic subsets}\label{global_subanalyticity}
The o-minimal structure of globally subanalytic sets is the one appearing in Assumption \ref{ass:targetMeasure}. It is the smallest structure which contains all restricted analytic functions, a formal definition is given below.
\begin{defn}
	A function $f:\mathbb{R}^{n}\longrightarrow \mathbb{R}$ is a restricted analytic function if there exists a function $F$ analytic on an open set containing $[-1,1]^{n}$, such that $f=F$ on $[-1,1]^n$ and $f=0$ otherwise. 
\end{defn}
\begin{prop}[\cite{van1986generalization}]
	\label{prop:globalSubanalytic}
	There exists an o-minimal structure which contains the graph of all restricted analytic functions. There is a smallest such structure denoted by $\mathbb{R}_{\text{an}}$. 
\end{prop}
An element of $\mathbb{R}_{\text{an}}$ is called globally subanalytic and definable functions in $\mathbb{R}_{\text{an}}$ are also called globally subanalytic. Note that any o-minimal structure contains the semi-algebraic subsets, thereby they are also globally subanalytic. The following lemma illustrates the versatility of globally subanalytic sets and allows to consider the probability densities in Example \ref{exmp_GSA}. It is a classical result (\cite[D.10]{driesmiller}) for which we provide a self-contained proof in Appendix \ref{sec:proofLemma} for completeness.

\begin{lem}\label{prop_exmp_GSA}
	Let  $U$ be an open subset of $\mathbb{R}^{n}$, $f:U\longrightarrow \mathbb{R}$ be an analytic function, and $K\subset U$ be a semi-algebraic compact subset of $U$. Then $f|_{K}: K \to \R$ is globally subanalytic.  
\end{lem}

\subsubsection{Inclusion of the exponential function}
The structure $\R_{\text{an}}$ contains the graph of the exponential function restricted to any bounded interval (because it is restricted analytic), but not on the whole real line. There is a bigger o-minimal structure which contains the whole graph of the exponential function, as the following shows.
\begin{prop}[\cite{Dries1995OnTR}]
	\label{prop:anExp}
	There exists an o-minimal structure which contains all functions definable in $\mathbb{R}_{\text{an}}$ and the graph of the exponential function. There is a smallest such structure denoted by $\mathbb{R}_{\text{an,exp}}$. 
\end{prop}
Note that $\mathbb{R}_{\text{an,exp}}$ also contains the graph of the logarithm function which is identical to that of the exponential function up to a symmetry.

\subsection{Definability of optimal transport quantization functionals}
\label{sec:definability_quantization}
We start with the description of a technical result on the integration of globally subanalytic functions \cite{Cluckers2009StabilityUI} and then describe how it applies to our semi-discrete quantization losses under Assumption \ref{ass:targetMeasure}.

\subsubsection{Integration of globally subanalytic functions and partial minimization}
O-minimal structures are not stable under integration in general. This question constitutes one of the great challenges of the field. It is therefore not possible to directly use definability properties of the objective functions $G_N$ in \eqref{eq:optimal_quant} and $F_N$ in \eqref{Objective_function}.
Several partial answers are known, one of them allowing to treat globally subanalytic integrands \cite{lion1998integration,comte2000nature,Cluckers2009StabilityUI}.
The following lemma is a direct consequence of \cite[Theorem 1.3]{Cluckers2009StabilityUI}, we provide a detailed argument for completeness.

\begin{lem}\label{integrale_def}
	Let $f:\mathbb{R}^{d}\times U\longrightarrow \mathbb{R}$ be a globally subanalytic function with $U\subset \R^n$ open (globally subanalytic). Define  $y\mapsto F(y):=\int_{\mathbb{R}^{d}}f(x,y)dx$ and suppose it is well defined for all $y$ in $U$. Then $F$ is definable in $\mathbb{R}_\text{an,exp}$.
\end{lem}

\begin{proof}
	By \cite[Theorem 1.3]{Cluckers2009StabilityUI}, if there exist $\left(f_{i}\right)_{i=1,\ldots,k}, (g_{i,j})_{\genfrac{}{}{0pt}{}{i=1,\ldots,k}{j=1,\ldots,l_{i}}}$ two families of globally subanalytic functions such that $f$ can be written
\begin{equation}\label{sum_product_log_restr_analy}
    (x,y)\mapsto f(x,y)=\sum_{i=1}^{k}f_{i}(x,y)\prod_{j=1}^{l_{i}}\log g_{i,j}(x,y),
\end{equation}
	then the parameterized integral $y\mapsto F(y)=\int_{\mathbb{R}^{d}}f(x,y)dx$ can also be written in the same form. By Lemma \ref{lem:definableComposition}, a function expressed as a sum and product of globally subanalytic functions and of the logarithm of globally subanalytic functions is definable in $\mathbb{R}_\text{an,exp}$.
\end{proof}

We conclude this section with the following classical result, which can be found in \cite[page 395]{ioffe2017variational} for semi-algebraic functions, but is still valid for general o-minimal structure. We provide a proof in Appendix \ref{sec:proofLemma} for completeness.

\begin{lem}\label{minmax} \hfill
\begin{enumerate}
    \item[(i)] Let $(f_{j})_{j\in J}$ be a finite collection of definable functions. Then $\underset{j\in J}{\min}\: f_{j}$ and $\underset{j\in J}{\max}\: f_{j}$ are definable functions. 
    \item[(ii)] Let $f:\mathbb{R}^{d}\times \mathbb{R}^{n}\longrightarrow \mathbb{R}$ be a definable function and $E$ be a definable subset of $\mathbb{R}^{n}$. Then, provided they are attained, $\underset{w\in E}{\min}\: f(\cdot,w)$ and $\underset{w\in E}{\max}\: f(\cdot,w)$ are definable functions.
\end{enumerate}
\end{lem}
\begin{rem}\label{rem_sup-inf}  
	Lemma \ref{minmax} remains true for the infimum and the supremum.
\end{rem}
\subsubsection{Definability of the optimal quantization functional}

Let us prove that the objective function $G_N$ in \eqref{eq:optimal_quant} is definable.

\begin{lem}\label{lem_G_definable}
	Under Assumption \ref{ass:targetMeasure},
	the optimal quantization function $G_{N}$ defined in \eqref{eq:optimal_quant} is definable in  $\mathbb{R}_{\text{an,exp}}$.
\end{lem}

\begin{proof}
	Let $f$ be the globally subanalytic function corresponding to the density of the target measure $\mu$. By stability of definability for the minimum function (Lemma \ref{minmax}), we have that the function $(x,Y)\mapsto \underset{i=1,\ldots,N}{\min}\|x-y_{i}\|^{2}f(x)$ is globally subanalytic. Since this function is integrable we can apply Lemma \ref{integrale_def} and obtain that $G_{N}$ in \eqref{eq:optimal_quant} is definable in  $\mathbb{R}_{\text{an,exp}}$.
\end{proof}

\subsubsection{Definability of the uniform quantization functional}
We prove that the objective function $F_N$ in \eqref{eq:unif_quantization} is definable. This is based on the representation in \eqref{Objective_G} and \eqref{Objective_function} for $Y\notin D_N$, see Remark \ref{rem:unifQuantEverywhere}. 
\begin{lem}\label{lem_F_definable}
	Under Assumption \ref{ass:targetMeasure}, the uniform quantization function $F_{N}$ defined in \eqref{Objective_function} is definable in  $\mathbb{R}_{\text{an,exp}}$.
\end{lem}

\begin{proof}
	Using Lemma \ref{integrale_def}, the function $\Phi$ defined in \eqref{Objective_G} is definable in  $\mathbb{R}_{\text{an,exp}}$. From Lemma \ref{minmax}, the function $F_N$ is also definable. Strictly speaking this argument works outside of the generalized diagonal, but the equality cases can be treated similarly as described in Remark \ref{rem:unifQuantEverywhere} and Theorem \ref{th:KantorovichDiagonal}.
\end{proof}

\subsubsection{Further elements on definability of quantization losses}
\label{sec:furtherElementsDefinability}
\label{subsec:rate_Lloyd}

The following remarks illustrate the importance of hypotheses in Lemma \ref{integrale_def} and its positioning in relation to potential alternatives from the o-minimal literature as well as potential extensions.

\paragraph{Potential subanaliticity of parametric integrals:}    
Lemma \ref{integrale_def} ensures the definability of the quantization functional in $\R_{\text{an,exp}}$ based on the log-analytic nature of parameterized integrals over $\mathbb{R}_{\text{an}}$. Beyond definability, one may wonder if, potentially under specific circumstances, one could obtain a more precise description of the parameterized integral as a subanalytic function. The following example suggests that this is not the case, even for constant densities on compact semi-algebraic sets. Set
\begin{align*}
    g(x,y,z) &= 
    \begin{cases}
        1 & \text{if } 0 \leq x \leq y \leq 1 \text{   and   } 0 \leq zy  \leq z \leq x\\
        0& \text{otherwise.}
    \end{cases}
\end{align*}
The function $g$ takes value $1$ on a compact semi-algebraic set and $0$ outside. We may integrate with respect to $y$ and $z$
\begin{align*}
    \forall x,y, \qquad \int_\R g(x,y,z) dz &= 
    \begin{cases}
        \frac{x}{y} & \text{if } 0 < x \leq y \leq 1 \\
        0& \text{otherwise.}
    \end{cases}\\
\\
    \forall x, \qquad \int_\R\int_\R g(x,y,z) dzdy &=  
    \begin{cases}
        -x \ln(x) & \text{if } 0 < x  \leq 1 \\
        0& \text{otherwise.}
    \end{cases}
\end{align*}
This function is not subanalytic nor semi-algebraic. If it were, it would have a polynomial growth around $0$. This example suggests that the definability in $\mathbb{R}_{\text{an,exp}}$ is rather tight without further geometric constraint on the support.

\paragraph{Power-like desingularizing function and convergence rates:}
It is natural to study the rates of convergence for the two Lloyd sequences. A classical approach consists in analyzing the desingularizing functions of $G_N$ and $F_N$ in the KL inequality. It is well known \cite{BolteAttouchcvproximalalgoanalyticfeatures} that if the desingularizing function is a power function of the form $t\mapsto t^{1-\theta}, \theta\in [0,1[$, which corresponds to the \L{}ojasiewicz's gradient inequality, then one can deduce asymptotic convergence rates of the gradient method from the exponent $\theta$. A particularly favorable case is $\theta=\frac{1}{2}$ as it ensures local linear convergence. From  \cite{kurdyka1994wf,bolte2007lojasiewicz}, we know that semi-algebraic and subanalytic functions verify the KL inequality with a power desingularizing function. However, our arguments only ensure the definability of the objective function in $\R_{\text{an,exp}}$. The previous example suggests that in general, even for constant densities on semi-algebraic sets, representation of parameterized integral may require the global logarithm, potentially preventing power like growth of desingularizing functions.

An approach extending these ideas was proposed in \cite{bolte2017error} to obtain complexity estimates for gradient methods. This abstract analysis was further explored recently in \cite{liu2024convergence}. In the latter work, the authors have two main growth conditions: functions of regular variation, which are essentially power like functions, and logarithmic error bounds, of the form $t \mapsto (- 1/\ln(t))^\gamma$ for a positive exponent $\gamma$ (\cite[Definition 5.8]{liu2024convergence}). For example, \cite[Section 5.1]{liu2024convergence} describes convergence rates for projection methods under this logarithmic growth condition. These techniques could in principle apply to our analysis of Lloyd sequences, since the desingularizing function $\varphi$ also encodes function growth around minimizers, see \cite{bolte2017error}. However, this would require to justify that desingularizing functions fit the logarithmic error bound condition which is unclear.  Indeed, the only knowledge we have about $\varphi$ is its definability in $\mathbb{R}_{\text{an,exp}}$ (see Lemma \ref{integrale_def} sand Section \ref{sec:furtherElementsDefinability}), which contains functions whose growth is beyond logarithmic. 

Similarly to \cite[Example 5.9]{liu2024convergence} the continuous, $\mathbb{R}_{\text{an,exp}}$-definable function
\begin{equation*}f \colon t \mapsto
\begin{cases}
        \exp\left(-\frac{1}{\exp\left( -\frac{1}{\mid t \mid} \right)}\right),& t \neq 0\\
        0,& t=0
    \end{cases}
\end{equation*}
does not satisfy any logarithmic error bound as above. Indeed $f( t ) = \phi^{-1}(\mid t \mid)$ with $\phi(t)=-\frac{1}{\log\left(-\frac{1}{\log(\mid t\mid)} \right)}$ for $t>0$. The function $\phi^{-1}$ is vanishing faster than any logarithmic error bound of the form given in \cite[Definition 5.8]{liu2024convergence}. As a consequence, the estimations given in this \cite[Theorem 5.12]{liu2024convergence} cannot be adapted systematically to our setting.   

We do note however that a possibly well suited preparation Theorem similar to the one in \cite{opris2021preparation} may allow to apply such arguments for special cases, this could be the object of future research.

\paragraph{Definability in Pfaffian closure:}
A famous result of Speissegger \cite{speissegger1999pfaffian} allows to characterize definability of certain integral related to Pfaffian functions. This would in principle allow to generalize Lemma \ref{integrale_def} to arbitrary o-minimal structure. However, this result does not fit the realm of parameterized integrals which is needed to treat quantization losses.
For example consider a parametrized integral of the form $F(x,y) = \int_0^y f(x,t)dt$ where $f$ is a definable function. We have
\begin{align*}
    \frac{d}{dx} F(x,y) & = \int_0^y \frac{d}{dx} f(x,t)dt,
    &\frac{d}{dy} F(x,y) = f(x,y),
\end{align*}
for which only the second equation fits the Pfaffian conditions of \cite{speissegger1999pfaffian}. Therefore the main result of \cite{speissegger1999pfaffian} does not allow to conclude about the definability of the function $F$ above.

\paragraph{Preparation theorems:}
Lemma \ref{integrale_def} is derived from \cite[Theorem 1.3]{Cluckers2009StabilityUI} which is primarily based on the preparation theorem of globally subanalytic functions \cite{parusinski2001preparation,miller2006preparation}. This theorem is then used to establish a preparation theorem for log-analytic functions, making it particularly well-suited for integration. In \cite{opris2021preparation} a new preparation theorem is given for $\mathbb{R}_{\text{an,exp}}$ definable functions : discarding technicalities, one can write (upon some definable cell) definable functions as
\begin{equation}\label{preparation_Ranexp}
    f(x,t)=a(t)\mid y_0(x,t)\mid^{q_0}...\mid y_r(x,t)\mid^{q_r}  \exp(c(x,t))u(x,t).
\end{equation}
Although an improvement to \cite{van2002minimal}'s result in which the unit $u$ in \eqref{preparation_Ranexp} is not described, this expression won't allow to generalize our results for two reasons. First, the idea of Clucker
and Miller’s proof is to factorize log-analytic functions as a product of separable functions in $x$ and $t$ and then leverage the linearity of the integral. To this end, they use properties of the logarithm to split functions in $x$ and in $t$ from the log part of the prepared function. In the case of \eqref{preparation_Ranexp} however, the log is replaced by an exponential
which is not factorizable into two terms depending separately on $x$ and $t$. Secondly in Clucker and Miller's preparation theorem the unit $u$ is an analytic function composed with a definable vector valued function whereas here it is a power series.

\section{Beyond semi-discrete $W_2$ losses}
\label{definability losses}

We extend the definability results of the previous section to more general \emph{semi-discrete loss functions} $F:(\R^d)^N\to \R$ defined by
\begin{equation}\label{Loss_function_general_distance}
    F(Y) \quad = \quad D\left(\mu,\frac{1}{N}\sum_{i=1}^{N}\delta_{y_{i}}\right),
\end{equation}
where $D$ denotes an optimal transport divergence, and $\mu$ is a probability measure. In particular we consider optimal transports with general cost (Section \ref{extension_general_cost}), the max-sliced Wasserstein distance (Section \ref{sec:sliced_max}) and the entropy regularized optimal transport problem (Section \ref{sec:entropy}). For all these cases, we prove definability in $\R_{\text{an,exp}}$ (introduced in Proposition \ref{prop:anExp}) of the resulting semi-discrete loss in \eqref{Loss_function_general_distance}. These results illustrate the relevance of the log analytic nature of integral of globally subanalytic objects in a semi-discrete optimal transport context.

Throughout this section, the measure $\mu$ is assumed to satisfy Assumption \ref{ass:targetMeasure} and we introduce a general cost function $c$ as follows.
\begin{assumption}
	\label{ass:costFunction}
    The cost function $c\colon \R^d \times \R^d \to \R_+$ is lower semicontinuous and globally subanalytic.
\end{assumption}
\begin{rem}
	The compacity of the support of $\mu$ is not strictly required and could be replaced by integrability conditions. Yet, for the sake of simplicity, we consider $\mu$ as in Assumption \ref{ass:targetMeasure}, especially to preserve an homogeneous set of assumptions throughout the text. Additionally, the results of this section directly hold replacing the uniform weights in \eqref{Loss_function_general_distance} by any fixed weights $(\pi_{1},...\pi_{N})\in \Delta_{N}$.
	\label{rem:integrability}
\end{rem}

\subsection{Kantorovich duality and definability for semi-discrete optimal transport losses with general costs}
\label{extension_general_cost}
The following is a reformulation of Kantorovich duality for optimal transport, specified for semi-discrete losses; see e.g. \cite[Theorem 5.9]{villanioptoldnew}. We pay special attention to ties in the Dirac mass support points.
Let $\mu$ and $c$ be as in Assumptions \ref{ass:targetMeasure} and \ref{ass:costFunction} and let $Y = (y_i)_{i=1}^N \in (\R^d)^N$ be a set of $N$ points in $\R^d$. We set 

\begin{align}
	\mathcal{T}_c \left(\mu, \frac{1}{N} \sum_{i=1}^N \delta_{y_i} \right) \quad = \quad \inf_{\gamma \in \Pi} \ \int_{\R^d \times \R^d} c(x,y) d\gamma (x,y),
	\label{eq:generalCostOptTransport}
\end{align}
where the infimum is taken over couplings (or plans) $\gamma\in\Pi$ whose first marginal is $\mu$ and whose second marginal is the discrete measure $\frac{1}{N}\sum_{i=1}^N \delta_{y_i}$. The loss functions in \eqref{opt_quant_0} and \eqref{eq:unif_quantization_0} correspond to the 2-Wasserstein distance squared for which the cost function $c$ is the squared Euclidean distance, i.e. $c(x,y) = \Vert x-y\Vert^2$ for $x,y\in\R^d$.

We let $\lambda \colon (\R^d)^N \to \R^N$ be such that for all $i = 1,\ldots, N$ and $Y\in(\mathbb{R}^d)^N$,
\begin{align}
	\lambda_i(Y) = \begin{cases}
		\frac{|\{j,\, j\geq i,\, y_j = y_i\}|}{N}& \text{ if } y_k \neq y_i\  \forall k<i, \\
		0&\text{ otherwise}.
	\end{cases}
	\label{eq:KantorovichDiagonalWeights}
\end{align}

For $Y \not \in D_N$, the generalized diagonal defined in \eqref{eq:generalizedDiagonal}, we have $\lambda_i(Y) = \frac{1}{N}$, $i = 1,\ldots, N$. When there are ties with equal points, $\lambda$ assign the mass of the whole group to the first index and zero to the remaining indices. Note that $\lambda_i \geq 0$ and $\sum_{i=1}^N \lambda_i = 1$ and the function $\lambda$ is semi-algebraic (it is constant on finitely many disjoint pieces defined by linear equalities and their complements). We set $g_c\colon (\R^d)^N \times \R^N\to\R$ as follows

\begin{align}
	g_c(Y,w) = \int_{\R^d} \ \min_{i:\, \lambda_i(Y)>0} \{c(x,y_i) - w_i\} d\mu(x) + \sum_{i=1}^N \lambda_i(Y) w_i.
	\label{eq:KantorovichDiagonalGc}
\end{align}

\begin{thm}
	Let $\mu$ be a probability measure on $\R^d$, $Y = (y_i)_{i=1}^N \in (\R^d)^N$ be a set of $N$ points in $\R^d$ and $c \colon \R^d \times \R^d \to \R_+$ be lower semicontinuous. Then 
	\begin{align*}
		\mathcal{T}_c \left(\mu, \frac{1}{N} \sum_{i=1}^N \delta_{y_i} \right) \quad =	\quad \max_{w \in \R^n} g_c(Y,w),
	\end{align*}
	where $\mathcal{T}_c$ is given in \eqref{eq:generalCostOptTransport} and $g_c$ is given in \eqref{eq:KantorovichDiagonalGc}.
	\label{th:KantorovichDiagonal}
\end{thm}
\begin{proof}
	This is a direct implication of \cite[Theorem 5.9]{villanioptoldnew}.
	The design of the function $\lambda$ in \eqref{eq:KantorovichDiagonalGc} ensures that 
	\begin{align*}
		\nu:=\frac{1}{N} \sum_{i=1}^N \delta_{y_i} = \sum_{i=1}^N \lambda_i(Y) \delta_{y_i},
	\end{align*}
	where, furthermore, the set of weights $\lambda_i(Y) > 0$ corresponds to distinct points $y_i$, all ties being merged. A direct inspection of the function $g_c$ in \eqref{eq:KantorovichDiagonalGc} shows that it does not depend on variables $w_i$ corresponding to zero weights $\lambda_i(Y) = 0$. Therefore we may ignore them and consider $l \leq N$ non-zero weights $\lambda_i(Y)$. Assuming that they correspond to the first $l$ weights (up to a permutation that keeps $g_c$ and $\mathcal{T}_c$ invariant), we obtain
	\begin{align*}
		\nu &= \sum_{i=1}^l \lambda_i(Y) \delta_{y_i},\\
		g_c(Y,w) &= \int_{\R^d} \ \min_{i=1,\ldots,l} \{c(x,y_i) - w_i\} d\mu(x) + \sum_{i=1}^l \lambda_i(Y) w_i,
	\end{align*}
	where the points $y_1,\ldots,y_l$ are pairwise distinct and form the support of the measure $\nu$. In this setting, we may interpret the dual variables $w_1,\ldots, w_l$ as a function $\phi \in L^1(\nu)$ by setting $\phi(y_i) = w_i$ for $i=1,\ldots,l$, since support points are pairwise distinct. The claimed result then follows from Kantorovich duality \cite[Theorem 5.9]{villanioptoldnew} by noticing that with this interpretation one has $\int \phi d\nu =  \sum_{i=1}^l \lambda_i(Y) w_i$ and $\min_{i=1,\ldots,l} \{c(x,y_i) - w_i\} = - \max_{y \in \{y_1,\ldots,y_l\}} \{\phi(y) - c(x,y)\} = -\phi^c(x)$, the $c$-transform of $\phi$.

\end{proof}

The following is our main definability result for this section. 
\begin{lem}\label{Wasserstein definable}
	Let $\mu$ be as in Assumption \ref{ass:targetMeasure} and $c$ be as in Assumption \ref{ass:costFunction}, then the loss function 
	\begin{equation}
		F_{c} : Y\in (\mathbb{R}^{d})^{N}\longmapsto \mathcal{T}_c\left(\mu,\frac{1}{N}\sum_{i=1}^{N}\delta_{y_{i}}\right)
	\end{equation}
	is definable in $\mathbb{R}_{\text{an,exp}}$ where $\mathcal{T}_c$ is defined in \eqref{eq:generalCostOptTransport}.
\end{lem}
\begin{proof}
	We aim at showing that the function $g_c$ appearing in the Kantorovich dual formulation in Theorem \ref{th:KantorovichDiagonal}, described in \eqref{eq:KantorovichDiagonalGc}, is definable and the result will follow from Lemma \ref{minmax}.

	First the function $\lambda$ is semi-algebraic as it is constant on finitely many subsets defined by finitely many linear equalities and their complements. Under Assumption \ref{ass:targetMeasure}, we have
	\begin{align*}
		g_c(Y,w) = \int_{\R^d} \min_{i:\, \lambda_i(Y)>0} \{c(x,y_i) - w_i\} f(x)dx + \sum_{i=1}^N \lambda_i(Y) w_i,
	\end{align*}
	where the density $f$ of $\mu$ and the cost $c$ are globally subanalytic. This property is preserved under finite minimization and composition (Lemma \ref{minmax} and Lemma \ref{lem:definableComposition}), and therefore the integrand in $g_c$ is a globally subanalytic function of $(x,Y,w)$.
	Lemma \ref{integrale_def} ensures that $g_c$ is definable in $\mathbb{R}_{\text{an,exp}}$. This concludes the proof. 
\end{proof}

\subsection{The sliced and max-sliced Wasserstein distances}
\label{sec:sliced_max}

\paragraph{Sliced Wasserstein.} For any $\theta\in\mathbb{S}^{d-1}$ the unit sphere of $\R^d$, we set $P_{\theta}:x\in\mathbb{R}^{d}\mapsto  \langle \theta,x \rangle$, the projection on the line directed by $\theta$, and we let $\sigma$ be the uniform probability measure over $\mathbb{S}^{d-1}$.
Given $\mu$ and $\nu$ two probability measures on $\mathbb{R}^{d}$, the sliced Wasserstein distance is defined as follows:  
\begin{equation}\label{sliced}
    \text{SW}_{2}^{2}(\mu,\nu)=\int_{\mathbb{S}^{d-1}}W_{2}^{2}({P_{\theta}}_{\sharp} \mu,{P_{\theta}}_{\sharp} \nu)d\sigma(\theta).
\end{equation}
Here $W_2^2$ still denotes the $2$-Wasserstein distance squared and corresponds to the loss in \eqref{eq:generalCostOptTransport} with $c(x,y) = (x-y)^2$ (the distance is taken between univariate measures).
We recall that the pushforward operator of a measure $\mu$ in $\R^d$ by a measurable map $T:\R^d\to\R^n$ is defined as the measure $T_{\sharp}\mu$ such that for all Borelian $B\subset\R^n, T_{\sharp} \mu(B)=\mu\left(T^{-1}(B)\right)$. We present definability of the sliced Wasserstein distance between two discrete probability measures. Note that even in the fully discrete setting, the divergence in \eqref{sliced} involves an integral over the sphere and therefore does not fall directly within the scope of definable functions.
\begin{lem}\label{lem_SW_discrete}
    Let $N,M\in \mathbb{N}^{*}$, then the following function is definable in $\mathbb{R}_{\text{an,exp}}$ 
\begin{align*}
	F_{\text{SW-discr}} :  (\mathbb{R}^{d})^{N}\times (\mathbb{R}^{d})^{M} \times \Delta_{N}\times \Delta_{M}  &\to \R \\
	(X,Y,a,b) &\mapsto \text{SW}_{2}^{2}\left(\sum_{i=1}^{N}a_{i}\delta_{x_{i}},\sum_{j=1}^{M}b_{j}\delta_{y_{j}}\right).
\end{align*}
\end{lem}
\begin{proof}
    By definition of the pushforward of a discrete measure by $P_\theta$, the integrand in the definition of $\text{SW}_{2}^{2}$ in \eqref{sliced} corresponds to
	\begin{align*}
		(\theta, X,Y,a,b) \mapsto W_2^2 \left(\sum_{i=1}^{N}a_{i}\delta_{\left\langle \theta,x_{i}\right\rangle},\sum_{j=1}^{M}b_{j}\delta_{\left\langle \theta, y_{j}\right\rangle}\right).
	\end{align*}
	This function is semi-algebraic, for example it can be written explicitly as a finite dimensional linear program with semi-algebraic data or using quantile functions for finitely many support points, which are semi-algebraic. The result follows from Lemma \ref{integrale_def}.
\end{proof}
The functional $F_{\text{SW-discr}}$ is studied in \cite{tanguy2023convergence} in the context of machine learning applications, where convergence of a stochastic gradient descent algorithm is described using the weakly convex nature of the resulting loss function. Unlike the case of optimal transport with general costs, the definability of semi-discrete sliced Wasserstein loss does not follow directly from \cite[Theorem 1.3]{Cluckers2009StabilityUI}. Indeed, the main result states that a certain class of log-analytic functions is stable under integration. However, log-analytic functions are not stable under maxima or absolute value, which would be required to conclude regarding the semi-discrete sliced Wasserstein loss. Still, it is reasonable to believe that the semi-discrete loss in \eqref{Loss_function_general_distance} is definable if $D$ is the sliced Wasserstein distance, but this result is out of the scope of the present work.

\paragraph{Max-sliced Wasserstein.}
Replacing the integral in \eqref{sliced} by a maximum over the unit sphere leads to the max-sliced Wasserstein problem \cite{kolouri2019generalized, deshpande2019max} defined as
\begin{equation}\label{max_sliced}
    \text{MSW}_{2}^{2}(\mu,\nu)=\underset{\theta \in \mathbb{S}^{d-1}}{\max} \ W_{2}^{2}({P_{\theta}}_{\sharp}\mu,{P_{\theta}}_{\sharp}\nu).
\end{equation}
In practice, it requires an estimator for the maximum (see the remark following the Claim 3 in \cite{deshpande2019max}). The following shows a definability result for the semi-discrete max-sliced Wasserstein distance.
\begin{lem}
	Let $\mu$ be as in Assumption \ref{ass:targetMeasure} and $c$ be as in Assumption \ref{ass:costFunction}, then the objective function 
	\begin{equation}
		F_{\text{MSW}}: Y\in (\mathbb{R}^{d})^{N}\longmapsto \text{MSW}_{2}^{2}\left(\mu,\frac{1}{N}\sum_{i=1}^{N}\delta_{y_{i}}\right)
	\end{equation}	
	is definable in $\mathbb{R}_{\text{an,exp}}$, where $\text{MSW}_{2}^{2}$ is defined in \eqref{max_sliced}.
\end{lem}
\begin{proof}
	For any $\theta\in \mathbb{S}^{d-1}$, we have that the pushforward of the discrete measure by $P_\theta$ is given by $\frac{1}{N}\sum_{i=1}^{N}\delta_{\langle y_{i},\theta \rangle}$. 
	Using Kantorovich duality (Theorem 5.9 in \cite{villanioptoldnew}), following the same construction as in Section \ref{extension_general_cost} we have, denoting by $Y_\theta$ the projection of points in $Y$ on the direction $\theta$
	\begin{align*}
		\text{MSW}_{2}^{2}\left(\mu,\frac{1}{N}\sum_{i=1}^{N}\delta_{y_{i}}\right)&=\underset{\theta\in \mathbb{S}^{d-1},\, w\in \mathbb{R}^{N}}{\max}\int_{\mathbb{R}}\underset{\lambda_i(Y_\theta) > 0}{\min}\left\{ (t-\langle y_{i},\theta \rangle)^{2}-w_{i} \right\}d{P_{\theta}}_{\sharp}\mu(t)+\sum_{i=1}^{N}\lambda_i(Y_\theta) w_{i}\\
	    &=\underset{\theta\in \mathbb{S}^{d-1},\, w\in \mathbb{R}^{N}}{\max}\int_{\mathbb{R}^{d}}\underset{\lambda_i(Y_\theta)> 0}{\min}\left\{ \langle x-y_{i},\theta \rangle^{2}-w_{i} \right\}f(x)dx + \sum_{i=1}^{N} \lambda_i(Y_\theta) w_{i},
	\end{align*}
	where the first equality follows from Theorem \ref{th:KantorovichDiagonal} and $\lambda$ is given as in \eqref{eq:KantorovichDiagonalWeights} for $d = 1$ applied to the projections $\left\langle \theta,y_i \right\rangle$, $i = 1,\ldots, N$.
	The second equality corresponds the change of variable for image measures (for example \cite[Theorem 3.6.1]{MeasuretheoryBogachev}). The integrand is globally subanalytic so by Lemma \ref{integrale_def}, the integral is definable in $\mathbb{R}_{\text{an,exp}}$, and so is the maximum by Lemma \ref{minmax}.

\end{proof}

\subsection{Entropic regularization}
\label{sec:entropy}

Entropic regularization is widely used in optimal transport and can for example induce desirable computational and algorithmic features \cite{cuturi2013sinkhorn} (see also section 4.1 of \cite{peyre2019computational}).
This regularization was considered in a semi-discrete setting in \cite[Section 2]{Entropic}.
Given $\epsilon>0$ and a general cost function $c \colon \R^d \times\R^d \to \R_+$, the entropy regularized optimal transport loss $W_{\epsilon}$ between two probability measures $\mu$ and $\nu $ is defined as follows 
\begin{equation}\label{entropic_wass}
    W_{\epsilon}(\mu,\nu)=\underset{\gamma \in \Pi}{\min}\int_{\R^d\times \R^d}c(x,y)d\gamma(x,y)+\epsilon\: \text{H}(\gamma\vert\mu \otimes \nu)
\end{equation}
where $\Pi$ denotes the set of measures on $\R^d \times \R^d$ whose first and second marginals are $\mu$ and $\nu$, and $\mu \otimes \nu$ is the product measure. The function H denotes the Kullback-Leibler divergence, which is defined for probability measures $\gamma, \zeta$ on $\R^d \times \R^d$ by: 
\begin{equation*}
    \text{H}(\gamma\vert\zeta)=\int_{\R^d\times \R^d}\left[\log\left(\frac{d\gamma}{d\zeta}(x,y)\right)-1 \right]d\gamma(x,y)
\end{equation*}
and $\frac{d\gamma}{d\zeta}$ denotes the relative density of $\gamma$ with respect to $\zeta$ (the divergence is infinite if $\gamma$ is not absolutely continuous with respect to $\zeta$). 

\begin{lem}
	\label{lem:entropic}
	Let $\mu$ be as in Assumption \ref{ass:targetMeasure} and $c$ be as in Assumption \ref{ass:costFunction}. Suppose that $R>0$ is such that $\supp \mu \subset B_{d}(0,R)$, the unit ball of radius $R$, and $c$ is in addition analytic on an open set containing $B_{d}(0,R)\times B_{d}(0,R)$. 
	Then given a fixed $\epsilon > 0$, the entropy regularized transport functional (restricted to the $R$-radius ball)
	\begin{equation}
	F_{\epsilon}: Y\in B_{d}(0,R)^{N}\longmapsto W_{\epsilon}\left(\mu,\frac{1}{N}\sum_{i=1}^{N}\delta_{y_{i}}\right)
	\end{equation}
 is definable in $\mathbb{R}_{\text{an,exp}}$.
\end{lem}

\begin{proof}
Using \cite[Proposition 2.1]{Entropic}, we have that $W_{\epsilon}$ can be written in dual form

\begin{align}
	W_{\epsilon}\left(\mu, \frac{1}{N} \sum_{i=1}^N \delta_{y_i}\right)&=\underset{w\in \mathcal{C}(B_{d}(0,R))}{\max}\quad\int_{\R^{d}} -\epsilon\log\left(\frac{1}{N} \sum_{i=1}^N \exp\left(\frac{w(y_i)-c(x,y_i)}{\epsilon}\right)\right)d\mu(x) \nonumber\\
	&\qquad + \frac{1}{N} \sum_{i=1}^N w(y_i) -\epsilon \label{S_eps}\tag{$S_{\epsilon}$} 
\end{align}
	where the soft-max approximation of the $c$-transform of $w$ appears and the maximum is over continuous functions. We first remark that the addition of a constant to $w$ in \eqref{S_eps} does not change the value of the objective, so we may assume for example that $w(y_1) = 0$, without modifying the value of the maximum. Furthermore, it is known that since $c$ is $\mathcal{C}^\infty$ by assumption, the optimal potential $w$ in \eqref{S_eps} is Lipschitz continuous with the same Lipschitz constant as the cost $c$, say on $B_{d}(0,R)$ \cite[Proposition 1]{genevay2019sample}. More precisely, it can be deduced from the smooth $c$-transform formulation of the dual variable and a first order condition. Hence the potential $w$ in \eqref{S_eps} may be assumed to be uniformly bounded by a constant $K$ which only depends on $c$ and $R$. Now we may use the same device as in Theorem \ref{th:KantorovichDiagonal} and obtain

\begin{align}
	W_{\epsilon}\left(\mu, \frac{1}{N} \sum_{i=1}^N \delta_{y_i}\right)&=\underset{\|w\|_\infty \leq K, w_1 = 0}{\max}\quad\int_{\R^{d}} -\epsilon\log\left(\sum_{i=1}^N \lambda_i(Y) \exp\left(\frac{w_i-c(x,y_i)}{\epsilon}\right)\right)d\mu(x) \nonumber\\
	&\qquad +  \sum_{i=1}^N \lambda_i(Y)w_i -\epsilon \label{eq:dualEntropicRegularized} 
\end{align}
where $\lambda$ is the semi-algebraic function in \eqref{eq:KantorovichDiagonalWeights} which allows to break ties.
The representation in \eqref{eq:dualEntropicRegularized} involves the logarithm restricted to a compact interval of the form $[a,b]$ where $a>0$, and the exponential function also restricted to a compact interval. Both are restricted analytic and hence globally subanalytic as in Proposition \ref{prop:globalSubanalytic}. Definable functions are stable by composition as described in Lemma \ref{lem:definableComposition}, therefore the expression in \eqref{eq:dualEntropicRegularized} involves the integral of a globally subanalytic function, and by Lemma \ref{lem:definableComposition} it is definable in $\R_{\text{an,exp}}$. We conclude using the stability of definable functions under maxima, Lemma \ref{minmax}.
\end{proof}

\begin{rem}
	\label{rem:epsilon}
	The result in Lemma \ref{lem:entropic} also holds if $\epsilon$ is considered as a variable, restricted to a compact interval of $\R_+^*$.
\end{rem}

\section{Conclusion}

The rigidity of o-minimal structures, and the log-analytic nature of globally subanalytic integrals allowed to prove sequential convergence of Lloyd-type algorithms under definability assumptions of the density of the target measure $\mu$. A by-product of the analysis is the introduction of the powerful tools of o-minimal geometry in a semi-discrete optimal transport context. The main assumption on the target measure $\mu$ is that it stems from a globally subanalytic density. While the class of such densities is extremely broad, it is not clear how restrictive this assumption is and how it relates to practice. For example, it seems difficult to practically distinguish between absolutely continuous measures for which the global subanaliticity assumption on the density holds or not. Finally, exploring the nature of critical points resulting from Lloyd's iterations is a relevant next step (global/local minima, saddle point) which would require a second order analysis.

\section*{Acknowledgments}

The idea of using KL inequality for integral functionals based on definability was proposed by to Jérôme Bolte along numerous discussions on the geometrical approach to optimization.
The authors thank the anonymous reviewers for insightful comments which greatly improved our manuscript.
This work benefited from financial support from the French government managed by the National Agency for Research under the France 2030 program, with the reference "ANR-23-PEIA-0004".
Edouard Pauwels acknowledges the support of Institut Universitaire de France (IUF), the AI Interdisciplinary Institute ANITI funding, through the French ``Investments for the Future -- PIA3'' program under the grant agreement ANR-19-PI3A0004, Air Force Office of Scientific Research, Air Force Material Command, USAF, under grant numbers FA8655-22-1-7012, ANR Chess (ANR-17-EURE-0010), ANR Regulia and ANR Bonsai. 

\appendix

\section{Additional proofs and technical results}
\label{sec:proofLemma}
The following lemma will be used throughout the text, see for example \cite[Exercise 1.11]{coste2000introduction}.
\begin{lem}
	The composition of two functions definable in the same o-minimal structure is definable.
	\label{lem:definableComposition}
\end{lem}

\begin{proof}[Proof of Lemma \ref{prop_exmp_GSA}]
	Let $\overline{B_{1,\infty}(\bar{x}_{1},r)},\:\ldots\:,\overline{B_{l,\infty}(\bar{x}_{l},r)}$ be a finite covering of $K$ by $l$ closed balls for the infinite norm with $r$ sufficiently small so that for all $i = 1,\ldots, l$, we have  $\overline{B_{i,\infty}(\bar{x}_{i},r)} \subset U$. Fix $i\in \{1, \ldots, l\}$ and consider the translation and scaling transformation 
\begin{equation*}
A_{i}: x \in\R^{n} \longmapsto \frac{1}{r}(x-\bar{x}_{i}) \in \mathbb{R}^{n}
\end{equation*}
so that $ A_{i}\left(\overline{B_{i,\infty}(\bar{x}_{i},r)}\right)=[-1,1]^{n}$.  These are affine bijections, so that $A_{i}(U)$ is an open set containing $[-1,1]^{n}$. By composition of analytic functions, we get that $f\circ A_{i}^{-1}:A_{i}(U)\longrightarrow \mathbb{R}$ is analytic. In particular, since $[-1,1]^{n}\subset A_{i}(U)$ we have that the restriction $f\circ A_{i}^{-1}|_{[-1,1]^{n}}$ is restricted analytic, hence globally subanalytic.

	Finally, by denoting $B_{i}:=\overline{B_{i,\infty}(\bar{x}_{i},r)}$, we have that $f|_{B_{i}}=f\circ A_{i}^{-1}|_{[-1,1]^{n}}\circ A_{i}$ where $A_{i}$ is affine, hence definable (its graph is a subspace). Therefore $f|_{B_{i}}$ is also globally subanalytic by Lemma \ref{lem:definableComposition}.
The graph of $f|_{K}$ can then be written: 
\begin{align*}
    \Gamma\left(f|_{K}\right)&=\Gamma\left(f|_{\bigcup_{i=1}^{l}B_{i}\cap K}\right) =\bigcup_{i=1}^{l} \Gamma\left(f|_{B_{i}\cap K}\right).
\end{align*}
	For all $i\in \{1,\ldots,N\}$, we therefore have $f|_{B_{i}\cap K}=f|_{B_{i}}\ \mathbbm{1}_{K}$, which is the product of two globally subanalytic functions as $\mathbbm{1}_{K}$ is the indicator function of a semi algebraic set.  Therefore $f|_{B_{i}\cap K}$ is globally subanalytic, for example using Lemma \ref{lem:definableComposition} for stability under product, and we conclude using the stability of globally subanalytic sets with finite unions.

\end{proof}

\begin{proof}[Proof of Lemma \ref{minmax}]

We prove the results (i) and (ii) for the minimum. The maximum case follows a symmetric proof using the hypograph.

First, we recall that by definition of the epigraph, $\epi \left( \underset{j\in J }{\min} \:f_{j}\right)=\bigcup_{j\in J}\epi(f_{j})$. We thus conclude for (i) since definable subsets are stable by finite unions.

For the second property (ii), we denote by $p:\mathbb{R}^{d+n+1}\longrightarrow \mathbb{R}^{d+1}$ the canonical projection on the first $d$ coordinates and the last coordinate. We aim at showing the following equality
\begin{equation*}
\epi\left(\underset{w\in E}{\min}\:f(.,w)\right)=p\left(\epi(f)\right).
\end{equation*}
Since the minimum is attained by hypothesis, we have
\begin{align*}
(x,y)\in \epi\left(\underset{w\in E}{\min}\: f(.\:,w)\right)
\quad\iff \quad & \underset{w\in E}{\min}\:f(x,w)\leq y\\
\iff \quad & \exists \ w'\in E\ \text{such that}\ f(x,w')\leq y\\
\iff \quad & (x,w',y)\in \epi(f)\\
\iff \quad & (x,y)\in \ p(\epi(f))\\
\end{align*}
which concludes the proof.

\end{proof}

\bibliographystyle{plain}
\bibliography{ref}

\begin{thebibliography}{10}

\bibitem{AbsMahAnd2005}
P-A. Absil, R.~Mahony, and B.~Andrews.
\newblock Convergence of the iterates of descent methods for analytic cost functions.
\newblock {\em SIAM J. Optim.}, 6(2):531--547, 2005.

\bibitem{asan2012introduction}
Umut Asan and Secil Ercan.
\newblock An introduction to self-organizing maps.
\newblock In {\em Computational intelligence systems in industrial engineering: With recent theory and applications}, pages 295--315. Springer, 2012.

\bibitem{BolteAttouchcvproximalalgoanalyticfeatures}
H.~Attouch and J.~Bolte.
\newblock On the convergence of the proximal algorithm for nonsmooth functions involving analytic features.
\newblock {\em Math. Program.}, 116:5--16, 2009.

\bibitem{proximalbolte2010defKL}
H.~Attouch, J.~Bolte, P.~Redont, and A.~Soubeyran.
\newblock Proximal alternating minimization and projection methods for nonconvex problems: An approach based on the {K}urdyka-{Ł}ojasiewicz inequality.
\newblock {\em Mathematics of Operations Research}, 35(2):438--457, 2010.

\bibitem{attouch2013convergence}
H.~Attouch, J.~Bolte, and B.F. Svaiter.
\newblock Convergence of descent methods for semi-algebraic and tame problems: proximal algorithms, forward--backward splitting, and regularized {G}auss--{S}eidel methods.
\newblock {\em Mathematical Programming}, 137(1):91--129, 2013.

\bibitem{ref_signal_opt_quant_2}
E.~Ayanoglu.
\newblock On optimal quantization of noisy sources.
\newblock {\em IEEE Transactions on Information Theory}, 36(6):1450--1452, 1990.

\bibitem{balzer2009capacity}
M.~Balzer, T.~Schl{\"o}mer, and O.~Deussen.
\newblock Capacity-constrained point distributions: A variant of {L}loyd's method.
\newblock {\em ACM Transactions on Graphics (TOG)}, 28(3):1--8, 2009.

\bibitem{bobkov2019one}
S.~Bobkov and M.~Ledoux.
\newblock One-dimensional empirical measures, order statistics, and {K}antorovich transport distances.
\newblock {\em American Mathematical Society}, 261(1259), 2019.

\bibitem{MeasuretheoryBogachev}
V.~Bogachev.
\newblock {\em Measure Theory}, volume~1.
\newblock Springer Berlin, Heidelberg, 2007.

\bibitem{bolte2007lojasiewicz}
J.~Bolte, A.~Daniilidis, and A.~Lewis.
\newblock The {{\L}}ojasiewicz inequality for nonsmooth subanalytic functions with applications to subgradient dynamical systems.
\newblock {\em SIAM Journal on Optimization}, 2007.

\bibitem{bolte2023subgradient}
J.~Bolte, T.~Le, and E.~Pauwels.
\newblock Subgradient sampling for nonsmooth nonconvex minimization.
\newblock {\em SIAM Journal on Optimization}, 33(4):2542--2569, 2023.

\bibitem{bolte2017error}
J.~Bolte, T.~P. Nguyen, J.~Peypouquet, and B.~W. Suter.
\newblock From error bounds to the complexity of first-order descent methods for convex functions.
\newblock {\em Mathematical Programming}, 165:471--507, 2017.

\bibitem{bolte2014proximal}
J.~Bolte, S.~Sabach, and M.~Teboulle.
\newblock Proximal alternating linearized minimization for nonconvex and nonsmooth problems.
\newblock {\em Mathematical Programming}, 146(1-2):459--494, 2014.

\bibitem{slicedradonbarycenter}
N.~Bonneel, J.~Rabin, G.~Peyr{\'e}, and H.~Pfister.
\newblock {Sliced and {R}adon {W}asserstein barycenters of measures}.
\newblock {\em {Journal of Mathematical Imaging and Vision}}, 1(51):22--45, 2015.

\bibitem{bottou1994convergence}
L.~Bottou and Y.~Bengio.
\newblock Convergence properties of the k-means algorithms.
\newblock {\em Advances in neural information processing systems}, 7, 1994.

\bibitem{Bournehexagonalpatterns}
D.~Bourne, M.~Peletier, and S.~Roper.
\newblock Hexagonal patterns in a simplified model for block copolymers.
\newblock {\em SIAM Journal on Applied Mathematics}, 74:1315--1337, 2014.

\bibitem{bouton1997multidimensional}
C.~Bouton and G.~Pag{\`e}s.
\newblock About the multidimensional competitive learning vector quantization algorithm with constant gain.
\newblock {\em The Annals of Applied Probability}, pages 679--710, 1997.

\bibitem{Cluckers2009StabilityUI}
R.~Cluckers and D.J. Miller.
\newblock {Stability under integration of sums of products of real globally subanalytic functions and their logarithms}.
\newblock {\em Duke Mathematical Journal}, 156(2):311 -- 348, 2011.

\bibitem{comte2000nature}
G.~Comte, J-M Lion, and J-P Rolin.
\newblock Nature log-analytique du volume des sous-analytiques.
\newblock {\em Illinois Journal of Mathematics}, 44(4):884--888, 2000.

\bibitem{coste2000introduction}
M.~Coste.
\newblock {\em An introduction to o-minimal geometry}.
\newblock Istituti editoriali e poligrafici internazionali Pisa, 2000.

\bibitem{cottrell2018self}
Marie Cottrell, Madalina Olteanu, Fabrice Rossi, and Nathalie~N Villa-Vialaneix.
\newblock Self-organizing maps, theory and applications.
\newblock {\em Revista de Investigacion Operacional}, 39(1):1--22, 2018.

\bibitem{cuturi2013sinkhorn}
M.~Cuturi.
\newblock {S}inkhorn distances: Lightspeed computation of optimal transport.
\newblock {\em Advances in neural information processing systems}, 26, 2013.

\bibitem{de2012blue}
F.~De~Goes, K.~Breeden, V.~Ostromoukhov, and M.~Desbrun.
\newblock Blue noise through optimal transport.
\newblock {\em ACM Transactions on Graphics (TOG)}, 31(6):1--11, 2012.

\bibitem{refhess}
F.~de~Gournay, J.~Kahn, and L.~Lebrat.
\newblock {Differentiation and regularity of semi-discrete optimal transport with respect to the parameters of the discrete measure}.
\newblock {\em {Numerische Mathematik}}, 141(2):429--453, 2019.

\bibitem{deshpande2019max}
I.~Deshpande, Y.T. Hu, R.~Sun, A.~Pyrros, N.~Siddiqui, S.~Koyejo, Z.~Zhao, D.~Forsyth, and A.~Schwing.
\newblock {M}ax-{S}liced {W}asserstein distance and its use for {GAN}s.
\newblock In {\em Proceedings of the IEEE/CVF Conference on Computer Vision and Pattern Recognition}, pages 10648--10656, 2019.

\bibitem{Dries1995OnTR}
L.~Dries and C.~Miller.
\newblock On the real exponential field with restricted analytic functions.
\newblock {\em Israel Journal of Mathematics}, 92:427, 1995.

\bibitem{driesmiller}
L.~Dries and C.~Miller.
\newblock {Geometric categories and o-minimal structures}.
\newblock {\em Duke Mathematical Journal}, 84(2):497 -- 540, 1996.

\bibitem{LloydCVT}
Q.~Du, M.~Emelianenko, and L.~Ju.
\newblock Convergence of the {L}loyd algorithm for computing centroidal {V}orono{\"i} tessellations.
\newblock {\em SIAM Journal on Numerical Analysis}, 44(1):102--119, 2006.

\bibitem{Lloydoptcvonedimlogconcave2}
Q.~Du, V.~Faber, and M.~Gunzburger.
\newblock Centroidal {V}orono{\"i} tessellations: Applications and algorithms.
\newblock {\em SIAM Review}, 41(4):637--676, 1999.

\bibitem{nondegeneracyLloyd}
M.~Emelianenko, L.~Ju, and A.~Rand.
\newblock Nondegeneracy and weak global convergence of the {L}loyd algorithm in {${\mathbb{R}}^d $}.
\newblock {\em SIAM Journal on Numerical Analysis}, 46(3):1423--1441, 2008.

\bibitem{Lloydoptcvonedimlogconcave1}
P.E. Fleischer.
\newblock Sufficient conditions for achieving minimum distortion in a quantizer.
\newblock {\em IEEE Int. Conv. Rec.}, 1, 1964.

\bibitem{genevay2019sample}
A.~Genevay, L.~Chizat, F.~Bach, M.~Cuturi, and G.~Peyr\'{e}.
\newblock Sample complexity of {S}inkhorn divergences.
\newblock In {\em Proceedings of the Twenty-Second International Conference on Artificial Intelligence and Statistics}, volume~89 of {\em Proceedings of Machine Learning Research}, pages 1574--1583. PMLR, 2019.

\bibitem{Entropic}
A.~Genevay, M.~Cuturi, G.~Peyr\'{e}, and F.~Bach.
\newblock Stochastic optimization for large-scale optimal transport.
\newblock In {\em Advances in Neural Information Processing Systems}, volume~29. Curran Associates, Inc., 2016.

\bibitem{Bookquantization}
S.~Graf and H.~Luschgy.
\newblock {\em Foundations of Quantization for Probability Distributions}, volume 1730.
\newblock Springer Berlin, Heidelberg, 2000.

\bibitem{ioffe2017variational}
A.D. Ioffe.
\newblock Variational analysis of regular mappings.
\newblock {\em Springer Monographs in Mathematics. Springer, Cham}, 2017.

\bibitem{kaiser2013integration}
Tobias Kaiser.
\newblock Integration of semialgebraic functions and integrated nash functions.
\newblock {\em Mathematische Zeitschrift}, 275(1):349--366, 2013.

\bibitem{kieffer1982exponential}
J.~Kieffer.
\newblock Exponential rate of convergence for {L}loyd's method {I}.
\newblock {\em IEEE Transactions on Information Theory}, 28(2):205--210, 1982.

\bibitem{ref_signal_opt_quant_3}
A.~Kipnis, G.~Reeves, Y.C. Eldar, and A.J. Goldsmith.
\newblock Compressed sensing under optimal quantization.
\newblock In {\em 2017 IEEE international symposium on information theory (ISIT)}, pages 2148--2152. IEEE, 2017.

\bibitem{refnewton}
Q.~Kitagawa, J.~M\'{e}rigot and B.~Thibert.
\newblock Convergence of a {N}ewton algorithm for semi-discrete optimal transport.
\newblock {\em Journal of the European Mathematical Society}, 21, 2019.

\bibitem{kohonen1991self}
Teuvo Kohonen.
\newblock Self-organizing maps: Ophmization approaches.
\newblock In {\em Artificial neural networks}, pages 981--990. Elsevier, 1991.

\bibitem{kolouri2019generalized}
S.~Kolouri, K.~Nadjahi, U.~Simsekli, R.~Badeau, and G.~Rohde.
\newblock Generalized {S}liced {W}asserstein distances.
\newblock {\em Advances in neural information processing systems}, 32, 2019.

\bibitem{KurdykaarticleominimpliqueKL}
K.~Kurdyka.
\newblock On gradients of functions definable in o-minimal structures.
\newblock {\em Annales de l'Institut Fourier}, 48(3):769--783, 1998.

\bibitem{kurdyka1994wf}
K.~Kurdyka and A.~Parusinski.
\newblock Wf-stratification of subanalytic functions and the {\l}ojasiewicz inequality.
\newblock {\em Comptes rendus de l'Acad{\'e}mie des sciences. S{\'e}rie 1, Math{\'e}matique}, 1994.

\bibitem{larson1967optimum}
R.~Larson.
\newblock Optimum quantization in dynamic systems.
\newblock {\em IEEE Transactions on Automatic Control}, 12(2):162--168, 1967.

\bibitem{LebratGournay2D}
L.~Lebrat, F.~de~Gournay, J.~Kahn, and P.~Weiss.
\newblock Optimal transport approximation of 2-dimensional measures.
\newblock {\em SIAM Journal on Imaging Sciences}, 12(2):762--787, 2019.

\bibitem{lion1998integration}
J.-M. Lion and J.-P. Rolin.
\newblock Intégration des fonctions sous-analytiques et volumes des sous-ensembles sous-analytiques.
\newblock {\em Annales de l'Institut Fourier}, 48(3):755--767, 1998.

\bibitem{liu2024convergence}
T.~Liu and B.F. Louren{\c{c}}o.
\newblock Convergence analysis under consistent error bounds.
\newblock {\em Foundations of Computational Mathematics}, 24(2):429--479, 2024.

\bibitem{refLloyd}
S.~Lloyd.
\newblock Least squares quantization in {PCM}.
\newblock {\em IEEE Transactions on Information Theory}, 28(2):129--137, 1982.

\bibitem{loi2010lecture}
T.L. Loi.
\newblock Lecture 1: O-minimal structures.
\newblock In {\em The Japanese-Australian Workshop on Real and Complex Singularities: JARCS III}, volume~43, pages 19--31. Australian National University, Mathematical Sciences Institute, 2010.

\bibitem{Lojasiewicz1963propriete}
S.~Lojasiewicz.
\newblock Une propriété topologique des sous-ensembles analytiques réels.
\newblock {\em Les équations aux dérivées partielles}, 117:87--89, 1963.

\bibitem{nonassymptotic}
Q.~M\'{e}rigot, F.~Santambrogio, and C.~Sarrazin.
\newblock Non-asymptotic convergence bounds for {W}asserstein approximation using point clouds.
\newblock In {\em Advances in Neural Information Processing Systems}, volume~34, pages 12810--12821. Curran Associates, Inc., 2021.

\bibitem{miller2006preparation}
D.~Miller.
\newblock A preparation theorem for {W}eierstrass systems.
\newblock {\em Transactions of the american mathematical society}, 358(10):4395--4439, 2006.

\bibitem{opris2021preparation}
A.~Opris.
\newblock On preparation theorems for ran, exp-definable functions.
\newblock {\em Journal of Logic and Analysis}, 15, 2021.

\bibitem{pages2015introduction}
G.~Pag{\`e}s.
\newblock Introduction to vector quantization and its applications for numerics.
\newblock {\em ESAIM: proceedings and surveys}, 48:29--79, 2015.

\bibitem{cvlloydanydimpages}
G.~Pag\`{e}s and J.~Yu.
\newblock Pointwise convergence of the {L}loyd algorithm in higher dimension.
\newblock {\em SIAM Journal on Control and Optimization}, 54(5):2354--2382, 2016.

\bibitem{pages1998space}
Gilles Pag{\`e}s.
\newblock A space quantization method for numerical integration.
\newblock {\em Journal of computational and applied mathematics}, 89(1):1--38, 1998.

\bibitem{parusinski2001preparation}
A.~Parusi{\'n}ski.
\newblock On the preparation theorem for subanalytic functions.
\newblock In {\em New developments in singularity theory}, pages 193--215. Springer, 2001.

\bibitem{peyre2019computational}
G.~Peyr{\'e} and M.~Cuturi.
\newblock Computational optimal transport: With applications to data science.
\newblock {\em Foundations and Trends{\textregistered} in Machine Learning}, 11(5-6):355--607, 2019.

\bibitem{rabin2012wasserstein}
J.~Rabin, G.~Peyr{\'e}, J.~Delon, and M.~Bernot.
\newblock {W}asserstein barycenter and its application to texture mixing.
\newblock In {\em Scale Space and Variational Methods in Computer Vision: Third International Conference, SSVM}, pages 435--446. Springer, 2012.

\bibitem{booksantambrogio}
F.~Santambrogio.
\newblock Optimal transport for applied mathematicians.
\newblock {\em Birk{\"a}user, NY}, 55(58-63):94, 2015.

\bibitem{speissegger1999pfaffian}
Patrick Speissegger.
\newblock The {P}faffian closure of an o-minimal structure.
\newblock {\em Journal für die reine und angewandte Mathematik}, 1999(508):189--211, 1999.

\bibitem{ref_signal_opt_quant}
J.~Z. Sun and V.~K. Goyal.
\newblock Optimal quantization of random measurements in compressed sensing.
\newblock In {\em 2009 IEEE International Symposium on Information Theory}, pages 6--10. IEEE, 2009.

\bibitem{tanguy2023convergence}
E.~Tanguy.
\newblock Convergence of {SGD} for training neural networks with {S}liced {W}asserstein losses.
\newblock {\em Transactions on Machine Learning Research}, 2023.

\bibitem{van1986generalization}
L.~{Van den Dries}.
\newblock A generalization of the tarski-seidenberg theorem, and some nondefinability results.
\newblock {\em Bulletin of the American Mathematical Society}, 15(2):189--193, 1986.

\bibitem{van2002minimal}
L.~van~den Dries and P.~Speissegger.
\newblock O-minimal preparation theorems.
\newblock {\em Model theory and applications}, 11:87--116, 2002.

\bibitem{villanioptoldnew}
C.~Villani.
\newblock {\em Optimal Transport: Old and New}, volume 338 of {\em Grundlehren der mathematischen Wissenschaften}.
\newblock Springer Berlin Heidelberg, 2008.

\bibitem{refvillanitopics}
C.~Villani.
\newblock {\em Topics in optimal transportation}, volume~58.
\newblock American Mathematical Soc., 2021.

\bibitem{Lloydoptimallocalcv}
X.~Wu.
\newblock On convergence of {L}loyd's method {I}.
\newblock {\em IEEE Transactions on Information Theory}, 38(1):171--174, 1992.

\bibitem{zhou2018adaptive}
Y.~Zhou, S.-M. Moosavi-Dezfooli, N.-M. Cheung, and P.~Frossard.
\newblock Adaptive quantization for deep neural network.
\newblock {\em Proceedings of the AAAI Conference on Artificial Intelligence}, 32(1), 2018.

\end{thebibliography}

\end{document}